\documentclass[12pt]{amsart}

\usepackage{amssymb, amsthm,multirow}

\newtheorem{theorem}{Theorem}[section]
\newtheorem{corollary}[theorem]{Corollary}

\newtheorem{lemma}[theorem]{Lemma}

\def\irr#1{{\rm  Irr}(#1)}

\def\cd#1{{\rm  cd}(#1)}

\def\ker#1{{\rm ker} (#1)}

\def\phi{\varphi}
\def\nl#1{{\rm nl}(#1)}

\begin{document}

\title[Centers of Characters]{Groups where the centers of the irreducible characters form a chain}
\author[Mark L. Lewis]{Mark L. Lewis}

\address{Department of Mathematical Sciences, Kent State University, Kent, OH 44242}
\email{lewis@math.kent.edu}


\begin{abstract}
We consider groups where the centers of the irreducible characters form a chain.  We obtain two alternate characterizations of these groups, and we obtain some information regarding the structure of these groups.  Using our results, we are able to classify those groups where the kernels of the irreducible characters form a chain.  We show that a result of Nenciu regarding nested GVZ groups is really a result about nested groups.  We obtain an alternate proof of a theorem of Isaacs regarding the existence of $p$-groups with a given set of irreducible character degrees.
\end{abstract}

\maketitle

\section{Introduction}

In this paper, all groups will be finite and we will write $\irr G$ for the set of irreducible characters of $G$.  We consider a problem suggested by Berkovich in \cite{berk}.  In Problem \# 24 in the section listed as Research Problems and Themes I of \cite{berk}, Berkovich suggested to study the $p$-groups where the kernels of irreducible characters form a chain and alternately, the $p$-groups where the set of centers of the irreducible characters form a chain.  We will see that it is quite easy to see that if $G$ is a $p$-group where the kernels of the irreducible characters of $G$ form a chain, then $G$ must be cyclic.

It is slightly more interesting to remove the hypothesis that $G$ is a $p$-group, and to look at any finite group where the kernels of the irreducible characters form a chain.  Even in this case, we will see that it is still not difficult to show the following:

\begin{theorem} \label{main1}
Let $G$ be a group.  Then the following are equivalent:
\begin{enumerate}
\item The kernels of the irreducible characters form a chain.
\item The normal subgroups of $G$ form a chain.
\item $G$ has a unique chief series.
\end{enumerate}	
\end{theorem}

Thus, it makes sense to modify Berkovich's question by ignoring the linear characters, and just assume that the kernels of the nonlinear irreducible characters form a chain.  Qian and Wang characterize the $p$-groups with this property in \cite{qiwa}.  In particular, if $G$ is a $p$-group for some prime $p$, then the kernels of the nonlinear irreducible characters of $G$ form a chain if and only if either $G$ has maximal class or $G'$ has order $p$ and $Z(G)$ is cyclic.





Again, it makes sense to ask this question for any group.   We obtain the following theorem.

\begin{theorem} \label{main2}
Let $G$ be a group.  Then the following are equivalent:
\begin{enumerate}
\item The kernels of the nonlinear irreducible characters of $G$ form a chain.
\item For every normal subgroup $M$ of $G$ with $1 \le M < G'$ either (a) $G'/M$ is not a chief factor of $G$ and $G/M$ has a unique minimal normal subgroup or (b) $G'/M$ is a chief factor of $G$ and $G/M$ has at most two minimal normal subgroups.
\end{enumerate}
\end{theorem}

For any group, we will see that if the kernels of the nonlinear irreducible characters form a chain, then the centers of the irreducible characters of $G$ form a chain.  In this situation, we will write $G = X_0 > X_1 > \dots X_n \ge 1$ for the distinct centers of irreducible characters, and we will call this the {\it chain of centers} for $G$.  Recall that Berkovich also suggested studying these groups in \cite{berk}.  It turns out that studying groups where the centers of the irreducible characters form a chain underlies our work and has much independent interest.  We will say that a group $G$ is {\it nested} if for all characters $\chi, \gamma \in \irr G$ either $Z(\chi) \le Z(\gamma)$ or $Z (\gamma) \le Z(\chi)$.  

In the next theorem, we present two characterizations of the nested groups that are independent of the irreducible characters.  

\begin{theorem} \label{main3}
Let $G$ be a group.  Then the following are equivalent:
\begin{enumerate}
\item $G$ is a nested group.
\item $G$ has a chain of normal subgroups $G = X_0 > X_1 > \dots > X_n \ge 1$ so that every normal subgroup $N$ of $G$ satisfies $Z (G/N) = X_i/N$ for some $i \in \{ 0, \dots, n \}$.
\item For all normal subgroups $N$ and $M$, if we define $Z_N/N = Z (G/N)$ and $Z_M/M = Z(G/M)$, then either $Z_N \le Z_M$ or $Z_M \le Z_N$.
\end{enumerate}
Furthermore, in (2), if for every $i \in \{ 0, \dots, n \}$ there exists a normal subgroup $N$ so that $Z(G/N) = X_i/N$, then $G = X_0 > X_1 > \dots >X_n \ge 1$ is the chain of centers for $G$.
\end{theorem}

In the situation of Theorem \ref{main3} (2) where $G = X_0 > X_1 > \dots > X_n \ge 1$ is the chain of centers for $G$, we will say that $n$ is the {\it nested length} of $G$.   We now produce a structure result on the factors of the chain of centers.

\begin{theorem} \label{main4}
Let $G$ be a nested group with chain of centers $G = X_0 > X_1 > \dots > X_n \ge 1$. Then the following are true:
\begin{enumerate}
\item For all $i = 1, \dots n$, one of the following occurs: 
\begin{enumerate}
	\item $[X_{i-1},G] \le X_i$ and both $X_{i-1}/X_i$ and $[X_{i-1},G]/[X_i,G]$ are elementary abelian $p$-groups for some prime $p$ or 
	\item $[X_{i-1},G]X_i/X_i$ is the unique minimal normal subgroup of $G/X_i$. 
\end{enumerate}
\item If there exist integers $i$ and $j$ with $1 \le i \le j \le n$ such that $[X_{i-1},G] \le X_j$, then both $X_{i-1}/X_j$ and $[X_{i-1},G]/[X_j,G]$ are elementary abelian $p$-groups for some prime $p$.
\item If there exist integers $1 \le j \le k \le n$ so that $[X_{i-1},G] \le X_i$ for all $i$ with $j \le i \le k$, then $X_{j-1}/X_k$ and $[X_{j-1},G]/[X_k,G]$ are $p$-groups for some prime $p$.
\item $G$ is nilpotent if and only if $[X_{i-1},G] \le X_i$ for all $i = i, \dots, n$.
\end{enumerate}  
\end{theorem}

We will show that a nilpotent nested group is the direct product of a nested $p$-group with an abelian group.  We also show that the factors of the upper central series except for $Z(G)$ and of the lower central series except for $G/G'$ of a nested group $G$ are elementary abelian $p$-groups for some prime $p$.

In her study of GVZ groups in \cite{gvz} and \cite{nested}, Nenciu has introduced nested groups in the context of GVZ-groups.  Nenciu introduced GVZ-groups in \cite{gvz}.  A group $G$ is a GVZ-group if every irreducible character $\chi$ vanishes on $G \setminus Z(\chi)$.  It was proved in \cite{gvz}
that GVZ groups are nilpotent.  

In the definition of nested groups in \cite{gvz} and \cite{nested}, Nenciu includes the fact that the containment of the centers depends on the cardinality of the degree of the irreducible character.  We include Nenciu's definition of ``nested'' under the name ``nested by degrees'' in Section \ref{degs}.  We will also define when a group is strongly nested by degrees.  We will show in Theorem \ref{strict} that these three definitions are equivalent for GVZ groups, and we will then show that a couple of the results that Nenciu proved for nested GVZ-groups in \cite{nested} are really just consequences of being nested.

It is not difficult to see that every group of nilpotence class $2$ is a GVZ-group.  Hence, we can use the results regarding nested GVZ-groups to obtain the following theorem for nested $p$-groups.  Recall that a $p$-group $G$ is semi-extraspecial if for all subgroups $N$ of index $p$ in $Z (G)$, the quotient $G/N$ is an extraspecial group, and $G$ is ultraspecial if $G$ is semi-extraspecial and $|G:Z (G)| = |Z(G)|^2$. 

\begin{theorem}\label{main5}
If $G$ is a nonabelian nested $p$-group with chain of centers $G = X_0 > X_1 > \dots X_n > 1$, then
\begin{enumerate}
\item $|G:X_1|$ is a square.
\item If $G' \le X_i$ for some integer $i = 1, \dots, n$, then $|G:X_j|$ is a square for all integers $1 \le j \le i$.
\item If the subgroup $N$ satisfies $[X_1,G] \le N < NG' = X_1$, then $G/N$ is a semi-extraspecial group.
\item If $X_2 < X_2G'$, then $|G:X_1| = |X_2G':X_2|^2 = |G':[X_1,G]|^2$ and $X_2 \cap G' = [X_1,G]$ and there exists a subgroup $A$ with $X_2 \le A < X_1$ so that $G/A$ is ultraspecial.
\end{enumerate}
\end{theorem}

We will include some additional results regarding nested $p$-groups of nilpotence class $2$.  We conclude this introduction with a result that recovers a theorem of Isaacs from \cite{isap}.  The main result of that paper is that if $p$ is a prime and ${\mathcal A}$ is a set of powers of $p$ that contain $1$, then there is a $p$-group $P$ of nilpotence class $2$ with $\cd P = {\mathcal A}$.  We now show that $P$ can be chosen to have the additional property that $P$ is nested by degrees.

\begin{theorem}\label{main6}
If ${\mathcal A}$ is a set of powers of a prime $p$ containing $1$, then there exists a group $G$ so that $G$ is a $p$-group of nilpotence class $2$, $G$ is nested by degrees, and $\cd G = {\mathcal A}$.
\end{theorem}

We would like to thank Adriana Nenciu for providing us a copy of \cite{nested} when it was a preprint and for several helpful conversations while we were writing this paper.  We would also like to thank Shawn Burkett for his work with Magma on checking examples, for helping format the table, and for a number of useful discussions while writing the paper.

%



\section{Nested groups}

In this section, we are going to explore nested groups.  We say that $G$ is a {\it nested group} if for all characters $\chi, \psi \in \irr G$ either $Z (\chi) \le Z (\psi)$ or $Z(\psi) \le Z (\chi)$.  Note that if $N$ is a normal subgroup of $G$, then $G/N$ will also be a nested group.  In particular, the centers of the irreducible characters form a chain of subgroups.

\begin{lemma} \label{defn}
Let $G$ be a group.  Then $G$ is nested if and only if there is a chain of normal subgroups $G = X_0 > X_1 > \dots > X_n \ge 1$ so that for every character $\chi \in \irr G$ there exists an integer $0 \le i \le n$ so that $Z (\chi) = X_i$ and for each $0 \le i \le n$ there exists a character $\chi \in \irr G$ so that $X_i = Z (\chi)$.
\end{lemma}

\begin{proof}
Suppose $G$ is nested.  We construct the $X_i$'s inductively.  Define $X_0 = G$ and observe that if $\lambda \in \irr G$ is linear, then $Z(\lambda) = X_0$.  Take $\chi_0$ to be some linear character.  Assume $i \ge 0$ and that $X_0 > X_1 > \dots > X_i$ have been constructed and that we have fixed $\chi_i \in \irr G$ so that $Z(\chi_i) = X_i$.  If $X_i$ properly contains no $Z (\psi)$ for all $\psi \in \irr G$, then we take $n = i$, and we terminate the chain at $X_n = X_i$.  Thus, we may suppose that $X_i$ properly contains some $Z(\psi)$ for $\psi \in \irr G$. Among such $\psi$ choose one so that $Z(\psi)$ maximal in the set $\{ Z(\gamma) \mid \gamma \in \irr G, Z(\gamma) < X_i \}$.  Take $\chi_{i+1} = \psi$ and $X_{i+1} = Z (\psi)$.  Let $\gamma \in \irr G$.  We know that $Z (\gamma) \le Z_0$, so take $i$ maximal so that $Z(\gamma) \le Z_i$.  We see that $Z(\gamma) \le Z(\chi_i)$ and $Z(\gamma) \not\le Z(\chi_{i+1})$.  The fact that $G$ is nested implies that $Z(\chi_{i+1}) < Z(\gamma)$.  Since $\chi_{i+1}$ was chosen to be maximal among the proper centers of characters in $Z(\chi_i)$, we conclude that $Z(\gamma) = Z(\chi_i)$.  This proves the forward direction.  Notice that the converse is immediate.
\end{proof}

Observe that $[X_0,G] = [G,G] = G'$.  We now show that the last term in the chain of centers will be the center of $G$.

\begin{lemma} \label{zero}
Let $G$ be a nested group with chain of centers $G = X_0 > X_1 > \dots > X_n \ge 1$.   Then $X_n = Z (G)$.
\end{lemma}

\begin{proof}
We know that $Z (G) = \cap_{\chi \in \irr G} Z (\chi)$.  For each $\chi \in \irr G$, we have $Z (\chi) = X_i$ for some $i$, so $X_n \le X_i = Z(\chi)$.  This implies that $X_n \le Z(G)$.  On the other hand, there exists $\chi \in \irr G$ so that $Z (\chi) = X_n$, and so, $Z (G) \le Z (\chi) = X_n$.  We conclude that $X_n = Z(\chi)$.
\end{proof}

We now start to look at the commutator subgroups associated with the chain of centers.  This next lemma is an immediate consequence of the fact that $Z(\chi)/\ker {\chi} = Z(G/\ker {\chi})$.  Since it is helpful to have this form of the lemma, we include it here.

\begin{lemma}\label{twoa}
Let $G$ be a nested group with chain of centers $G = X_0 > X_1 > \dots > X_n \ge 1$ and consider a character $\chi \in \irr G$.  If $Z (\chi) = X_i$, then $[X_i,G] \le \ker{\chi}$.
\end{lemma}


We now obtain a generalization of Lemma \ref{zero} regarding the quotient group $G/[X_i,G]$ for each value of $i$ with $i = 0, \dots, n$.

\begin{lemma} \label{threea}
Let $G$ be a nested group with chain of centers $G = X_0 > X_1 > \dots > X_n \ge 1$.  Then $X_i/[X_i,G] = Z (G/[X_i,G])$ for $i = 0, \dots, n$.
\end{lemma}

\begin{proof}
Let $Z/[X_i,G] = Z (G/[X_i,G])$.  It is obvious that 
$$X_i/[X_i,G] \le Z(G/[X_i,G]),$$ 
so $X_i \le Z$.  On the other hand, there exists $\chi \in \irr G$ with $Z (\chi) = X_i$ and by Lemma \ref{twoa}, we have $[X_i,G] \le \ker {\chi}$.  This implies that $Z \le X_i$.  We conclude that $Z = X_i$
\end{proof}

We now can show that $[X_{i-1},G] > [X_i,G]$.

\begin{corollary} \label{threec}
Let $G$ be a nested group with chain of centers $G = X_0 > X_1 > \dots > X_n \ge 1$.  If $1 \le i \le n$, then $[X_i,G] < [X_{i-1},G]$.
\end{corollary}

\begin{proof}
If $[X_{i-1},G] = [X_i,G]$, then Lemma \ref{threea} would imply that 
$$
X_i/[X_i,G] = Z (G/[X_i,G]) = Z(G/[X_{i-1},G]) = X_{i-1}/[X_{i-1},G],
$$ 
and we would have $X_i = X_{i-1}$ which is a contradiction.
\end{proof}

The following lemma determines when the converse of Lemma \ref{twoa} holds.

\begin{lemma} \label{threeb}
Let $G$ be a nested group with chain of centers $G = X_0 > X_1 > \dots > X_n \ge 1$ and consider a character $\chi \in \irr G$.  Then $Z(\chi) = X_i$ if and only if $[X_i,G] \le \ker {\chi}$ and $[X_{i-1},G] \not\le \ker {\chi}$.
\end{lemma}

\begin{proof}
If $Z(\chi) = X_i$, then by Lemma \ref{twoa} we have $[X_i,G] \le \ker{\chi}$.  If $[X_{i-1},G] \le \ker {\chi}$, then $X_{i-1}/\ker {\chi} \le Z(G/\ker {\chi}) = Z(\chi)/\ker{\chi} = X_i/\ker {\chi}$, and we would have $X_{i-1} < X_i$, a contradiction.  Thus, $[X_{i-1},G] \not\le Z(\chi)$.

Conversely, suppose that $[X_i,G] \le \ker{\chi}$ and $[X_{i-1},G] \not\le \ker {\chi}$.  Observe that $[X_i,G] \le \ker{\chi}$ implies that $X_i \le Z(\chi)$.  If $Z(\chi) \ne X_i$, then the fact that centers form a chain implies that $X_{i-1} \le Z(\chi)$ which would yield $[X_{i-1},G] \le \ker{\chi}$ and since this is a contradiction, we conclude that $X_i = Z(\chi)$.
\end{proof}

If $N$ is a normal subgroup of $G$, then it is not difficult to see that $G/N$ will be nested.  Notice that if $X/N$ is a term in the chain of centers for $G/N$, then there is a character $\chi \in \irr {G/N}$ so that $Z(\chi)/N= X/N$, but we may view $\chi \in \irr G$ and so, $Z(G) = X_i$ for some integer $i$ with $ 0 \le i \le n$.  Hence, the chain of centers for $G/N$ will consist of quotients coming from a subchain of the chain of centers for $G/N$.  Using Lemma \ref{threeb}, we can determine which of the $X_i$'s occur in the chain of centers for $G/N$.  This can be thought of a version of Corollary \ref{threec} for quotients.

\begin{lemma}\label{ten}
Let $G$ be a nested with chain of centers $G = X_0 > X_1 > \dots > X_n \ge 1$, and let $N$ be a normal subgroup of $G$.  If $1 \le i \le n$, then there exists a character $\chi \in \irr {G/N}$ with $Z(\chi) = X_i$ if and only if $N[X_i,G] < N[X_{i-1},G]$.
\end{lemma}

\begin{proof}
Suppose first that there exists a character $\chi \in \irr {G/N}$ with $Z(\chi) = X_i$.  By Lemma \ref{threeb}, we know that $[X_i,G] \le \ker{\chi}$ and $[X_{i-1},G] \not\le \ker{\chi}$.  This implies that $N[X_i,G] \le \ker{\chi}$, and if $N[X_i,G] = N[X_{i-1},G]$, then it would follow that $[X_{i-1},G] \le \ker{\chi}$ which is a contradiction, so we must have $N[X_i,G] < N[X_{i-1},G]$.  Conversely, suppose that $N[X_i,G] < N[X_{i-1},G]$.  We can find a nonprincipal character 
$$\lambda \in \irr {N[X_i,G]/N[X_{i-1},G]},$$ 
and consider a character $\chi \in \irr {G \mid \lambda}$.  Observe that $N \le \ker{\lambda}$, and since $N$ is normal in $G$, we have $N \le {\rm Core}_G (\ker{\lambda}) \le \ker {\chi}$, so $\chi \in \irr {G/N}$.  Also, $[X_i,G] \le \ker{\lambda} \le \ker{\chi}$.  If $[X_{i-1},G] \le \ker {\chi}$, then $N[X_{i-1},G] \le \ker {\chi}$ and this contradicts $\lambda$ nonprincipal.  Thus, $[X_{i-1},G] \not\le \ker {\chi}$.  By Lemma \ref{threeb}, we conclude that $Z (\chi) = X_i$.
\end{proof}

\section{Characterizations of nested groups}

The goal of this section is to find alternate characterizations of nested groups.  We first show that the center of any quotient of a nested group will be the quotient of one of the $X_i$'s.

\begin{lemma} \label{sixa}
Let $G$ be a nested group with chain of centers $G = X_0 > X_1 > \dots > X_n \ge 1$.  If $N$ is a normal subgroup of $G$ such that $[X_i,G] \le N$ and $[X_{i-1},G] \not\le N$ for some $i \in \{ 0, \dots, n \}$, then $Z(G/N) = X_i/N$.
\end{lemma}

\begin{proof}
Let $Z/N = Z (G/N)$.  Observe that $Z = \cap_{\theta \in \irr {G/N}} Z (\theta)$.  Let $\theta \in \irr {G/N}$, and observe that $[X_i,G] \le N \le \ker{\theta}$.  Thus, $X_i \le Z (\theta)$.  Since this is true for all characters $\theta \in \irr {G/N}$, we have $X_i \le \cap_{\theta \in \irr {G/N}} Z(\theta) = Z$.  On the other hand, since $[X_{i-1},G] \not\le N$, we have $N[X_i,G] = N < N[X_{i-1},G]$.  By Lemma \ref{ten}, we see that $Z(\chi) = X_i$ for some $\chi \in \irr {G/N}$.  Observe that $Z \le Z(\chi) = X_i$, and so, $Z = X_i$.
\end{proof}

The following is a sufficient condition for a group to be nested.


\begin{lemma} \label{sixca}
If $G$ is a group with chain of normal subgroups $G = X_0 > X_1 > \dots > X_n \ge 1$ so that every normal subgroup $N$ of $G$ satisfies $Z (G/N) = X_i/N$ for some $i \in \{ 0, \dots, n \}$, then $G$ is a nested group. Furthermore, if for every $i \in \{ 0, \dots, n \}$ there exists a normal subgroup $N$ so that $Z(G/N) = X_i/N$, then $G = X_0 > X_1 > \dots >X_n \ge 1$ is the chain of centers for $G$.
\end{lemma}

\begin{proof}
Let $\chi \in \irr G$, and let $N = \ker{\chi}$.  From the hypotheses, $Z(\chi)/N = Z(G/N) = X_i/N$ for some $i \in \{ 0, \dots, n\}$.  It follows that $Z(\chi) = X_i$.  If $\psi \in \irr G$, then similarly, $Z (\psi) = X_j$ for some $j$.  Since $X_i \le X_j$ or $X_j \le X_i$, we see that $G$ is a nested group.  Notice that $X_0 = G = \ker {1_G}$.  Thus, we consider an integer $i$ with $1 \le i \le n$.  Suppose there exists a normal subgroup $N$ so that $Z(G/N) = X_i/N$.  It follows that $X_i = \cap_{\theta \in \irr {G/N}} Z (\theta)$.  If $\theta \in \irr {G/N}$, then $Z(\theta) \ge X_i$ and if $Z(\theta) \ne X_i$, then $X_{i-1} \le Z(\theta)$.  Furthermore, if $Z (\theta) \ne X_i$ for all $\theta \in \irr {G/N}$, then $X_{i-1} \le \cap_{\theta \in \irr {G/N}} Z (\theta) = X_i$ which contradicts $X_i < X_{i-1}$.  Therefore, there exists a character $\theta \in \irr {G/N} \subseteq \irr G$ so that $Z (\theta) = X_i$.  We conclude that $G = X_0 > X_1 > \dots >X_n \ge 1$ is the chain of centers for $G$.
\end{proof}

We close this section by proving Theorem \ref{main3}.

\begin{proof}[Proof of Theorem \ref{main3}]
The fact that (1) implies (2) follows from Lemmas \ref{defn} and \ref{sixa}.  Suppose (2).  Let $N$ and $M$ be normal subgroups of $G$.  By (2), we see that $Z_N = X_i$ and $Z_M = X_j$ for some integers $0 \le i,j \le n$.  It follows that $Z_N \le Z_M$ or $Z_M \le Z_N$.  Hence, (3) is proved.  Suppose (3).  If $\chi, \psi \in \irr G$, then $Z(\chi) = Z_{\ker \chi}$ and $Z (\psi) = Z_{\ker \psi}$.  By (3), it follows that $Z(\chi) \le Z(\psi)$ or $Z(\psi) \le Z(\chi)$, and this proves $G$ is nested.  Therefore, (1) is proved. 

We now suppose for every $i$ with $0\le i \le n$ that there exists a normal subgroup $N$ so that $Z(G/N) = X_i/N$.  Fix an integer $i$, and let $N$ be the normal subgroup so that $Z(G/N) = X_i/N$.  It follows that $X_i = \cap_{\theta \in \irr {G/N}} Z (\theta)$.  If $\theta \in \irr {G/N}$, then $Z(\theta) \ge X_i$ and if $Z(\theta) \ne X_i$, then $X_{i-1} \le Z (\theta)$.  Furthermore, if $Z (\theta) \ne X_i$ for all $\theta \in \irr {G/N}$, then $X_{i-1} \le \cap_{\theta \in \irr {G/N}} Z (\theta) = X_i$ which contradicts $X_i < X_{i-1}$.  Therefore, there exists a character $\theta \in \irr {G/N} \subseteq \irr G$ so that $Z (\theta) = X_i$.  We conclude that $G = X_0 > X_1 > \dots >X_n \ge 1$ is the chain of centers for $G$.
\end{proof}

\section{Factors of the chain of centers}

In this section, we prove the various results about the factors of the chain of centers for a nested group.  We begin with a result that is a consequence of Lemma \ref{sixa}.  This gives a restriction on the normal subgroups of nested groups.

\begin{lemma}\label{sixb}
Let $G$ be a nested group with chain of centers $G = X_0 > X_1 > \dots > X_n \ge 1$.  If $N$ is a normal subgroup, then for each $i = 1, \dots, n$, either $[X_{i-1},G] \le N$ or $N \le X_i$.
\end{lemma}

\begin{proof}
Suppose that $[X_{i-1},G] \not\le N$.  By Lemma \ref{zero}, we know that $[X_n,G] = 1$.  So there is a minimal positive integer $j$ so that $[X_j,G] \le N$, and observe that we must have $j \ge i$.  By the minimality of $j$, we $[X_{j-1},G] \not\le N$.  Applying Lemma \ref{sixa}, we have $Z(G/N) = X_j/N$.  This implies that $N \le X_j \le X_i$.
\end{proof}

We now obtain a key result about the structure of $G/X_i$.  Note that we are not assuming that $G$ is nested in this lemma.  However, in light of Lemma \ref{sixb}, we know that this result will apply when $G$ is nested.

\begin{lemma} \label{sixc}
Let $G$ be a group with chain of normal subgroups $G = X_0 > X_1 > \dots > X_n \ge 1$ so that every normal subgroup $N$ of $G$ satisfies either $[X_{i-1},G] \le N$ or $N \le X_i$ for all $i = 1, \dots, n$.  Then for all $i = 1, \dots n$, one of the following occurs: (1) $[X_{i-1},G] \le X_i$ or (2) $[X_{i-1},G]X_i/X_i$ is the unique minimal normal subgroup of $G/X_i$.
\end{lemma}

\begin{proof}
Suppose that $[X_{i-1},G] \not\le X_i$, and suppose $N/X_i$ is a minimal normal subgroup of $G/X_i$.  Since $N$ is not contained in $X_i$, we can see from the hypothesis that $[X_{i-1},G] \le N$.  Notice that $X_i < [X_{i-1},G]X_i \le N$, and so, the minimality of $N$ implies that $N = [X_{i-1},G] X_i$.  This proves the result.
\end{proof}

We now show that a nested group $G$ is nilpotent if and only if the $X_i$'s form a central series for $G$.  In particular, we show that if $G$ is a nilpotent nested group, then the nilpotence class of $G$ is at most the number of distinct $X_i$'s.

\begin{lemma} \label{sixd}
Let $G$ be a nested group with chain of centers $G = X_0 > X_1 > \dots > X_n \ge 1$.  Then $G$ is nilpotent if and only if $[X_{i-1},G] \le X_i$ for all $i = 1, \dots, n$.
\end{lemma}

\begin{proof}
If $[X_{i-1},G] \le X_i$ for $i = 1, \dots, n$, then the $X_i$'s form a central series for $G$, and thus, $G$ is nilpotent.  Conversely, suppose $G$ is nilpotent.  We have $[X_i,G] \le X_i$.  If $[X_{i-1},G] \not\le X_i$, then by Lemma \ref{sixa}, we have $Z(G/X_i) = X_i/X_i$.  This implies that $G/X_i$ is not nilpotent which contradicts $G$ nilpotent.  Thus, we must have $[X_{i-1},G] \le X_i$.
\end{proof}

In this next lemma, we see that the factors of a chain of centers in a nested group that satisfy $[X_{i-1},G] \not\le X_i$ are far from nilpotent.

\begin{lemma}
Let $G$ be a nested group with chain of centers $G = X_0 > X_1 > \dots > X_n \ge 1$.  If $[X_{i-1},G] \not\le X_i$ for some integer $i$ with $1 \le i \le n$, then $[X_{i-1},G] = [X_{i-1},G,G]$.
\end{lemma}

\begin{proof}
Let $N = [X_{i-1},G,G]$.  We know that $[X_{i-1},G]$ is normal in $G$, so $N \le [X_{i-1},G]$.  Suppose $N < [X_{i-1},G]$.  Observe that $[X_{i-1},G]/N \le Z(G/N)$.  We can find an integer $j \ge i$ so that $[X_j,G] \le N$ and $[Z_{j+1},G] \not\le N$.  By Lemma \ref{sixa}, we have $Z(G/N) = X_j/N$.  It follows that $[X_{i-1},G] \le X_j \le X_i$ which is a contradiction.  Thus, we have $[X_{i-1},G] = N$.
\end{proof}

We now show that the factors of a chain of centers in a nested group where $[X_{i-1},G] \le X_i$ have exponent $p$.

\begin{lemma} \label{nine}
If $G$ is a nested group with chain of centers $G = X_0 > X_1 > \dots > X_n \ge 1$ and $[X_{i-1},G] \le X_i$ for some $i$ with $1 \le i \le n$, then $X_{i-1}/X_i$ and $[X_{i-1},G]/[X_i,G]$ are elementary abelian $p$-groups for some prime $p$.
\end{lemma}

\begin{proof}
Without loss of generality, we may assume that $i = n$, so $[X_i,G] = 1$.  By Lemma \ref{zero}, $X_i = Z (G)$.  Since $X_i < X_{i-1}$, we see that $[X_{i-1},G] > 1$ by Corollary \ref{threec}.  Thus, we can find $N < [X_{i-1},G] \le X_i = Z(G)$ so that $|[X_{i-1},G]:N| = p$ for some prime $p$.  Since $X_{i-1}/X_i = (X_{i-1}/N)/(X_i/N)$ and $X_i/N = Z (G/N)$ by Lemma \ref{sixa}, to show $X_{i-1}/X_i$ is elementary abelian, we may assume that $N = 1$.  Thus, $|[X_{i-1},G]| = p$.  Let $x \in X_{i-1}$ and $y \in G$.  Then $[x^p,y] = [x,y]^p = 1$ since $[x,y] \in [X_{i-1},G] \le Z(G)$.  Since this is true for all $y \in G$, we deduce that $x^p \in Z(G) = X_i$.  Since $[X_{i-1},G] \le X_i$, we know that $X_{i-1}/X_i$ is abelian, so $X_{i-1}/X_i$ is an elementary abelian $p$-group.

We continue to assume that $[X_i,G] = 1$, but we no longer assume that $|[X_{i-1},G]| = p$.  Obviously, $[X_{i-1},G]$ is abelian and generated by elements of the form $[x,y]$ where $x \in X_{i-1}$ and $y \in G$.  Since $[x,y] \in Z(G)$ and $x^p \in X_i = Z(G)$, we have $[x,y]^p = [x^p,y] = 1$.  Thus, the generators of $[X_{i-1},G]$ have order $p$, so $[X_{i-1},G]$ is elementary abelian.
\end{proof}

We generalize Lemma \ref{nine} to case where the commutator goes down more than one step.

\begin{lemma} \label{nine2}
If $G$ is a nested group with chain of centers $G = X_0 > X_1 > \dots > X_n \ge 1$ and $[X_{i-1},G] \le X_j$ for integers $i$ and $j$ with $1 \le i \le j \le n$, then $X_{i-1}/X_j$ and $[X_{i-1},G]/[X_j,G]$ are elementary abelian $p$-groups for some prime $p$.
\end{lemma}

\begin{proof}
Without loss of generality, we may assume that $j = n$, so $[X_j,G] = 1$.  By Lemma \ref{zero}, $X_j = Z (G)$.  Since $X_j \le X_i < X_{i-1}$, we see that $[X_{i-1},G] > 1$ by Corollary \ref{threec}.  Thus, we can find $N < [X_{i-1},G] \le X_j = Z(G)$ so that $|[X_{i-1},G]:N| = p$ for some prime $p$.  By Lemma \ref{sixa}, we obtain $X_j/N = Z(G/N)$.  Since $[X_k,G]N = [X_k,G] < [X_{k-1},G] = [X_{k-1},G]N$ for all $k$ with $j \le k \le i$, we have by Lemma \ref{ten} that $X_{i-1}/N > X_i/N > \cdots X_j/N$ are consecutive terms in the chain of centers for $G/N$.  Thus, to show $X_{i-1}/X_j$ is elementary abelian, we may assume that $N = 1$.  Under the assumption that $N = 1$, we have $X_j = Z(G)$ and $|[X_{i-1},G]| = p$.  Let $x \in X_{i-1}$ and $y \in G$.  Then $[x^p,y] = [x,y]^p = 1$ since $[x,y] \in [X_{i-1},G] \le Z(G)$.  Since this is true for all $y \in G$, we deduce that $x^p \in Z(G) = X_j$.  Since $[X_{i-1},G] \le X_j$, we know that $X_{i-1}/X_j$ is abelian, so $X_{i-1}/X_j$ is an elementary abelian $p$-group.

We continue to assume that $[X_j,G] = 1$, but we no longer assume that $|[X_{i-1},G]| = p$.  Obviously, $[X_{i-1},G]$ is abelian and generated by elements of the form $[x,y]$ where $x \in X_{i-1}$ and $y \in G$.  Since $[x,y] \in Z(G)$ and $x^p \in X_i = Z(G)$, we have $[x,y]^p = [x^p,y] = 1$.  Thus, the generators of $[X_{i-1},G]$ have order $p$, so $[X_{i-1},G]$ is elementary abelian.  Translating back to $G$, this implies that $[X_{i-1},G]/[X_j,G]$ is an elementary abelian $p$-group.
\end{proof}

Now, consider the case where we have a sequence of factors that all have the condition that $[X_{i-1},G] \le X_i$.  We know that each factor is an elementary abelian $p_i$-group for some prime $p_i$.  We show that the $p_i$'s are in fact all the same prime.

\begin{lemma} \label{nine1}
Suppose $G$ is a nested group with chain of centers $G = X_0 > X_1 > \dots > X_n \ge 1$. If there exist integers $1 \le j \le k \le n$ so that $[X_{i-1},G] \le X_i$ for all $i$ with $j \le i \le k$, then $X_{j-1}/X_k$ and $[X_{j-1},G]/[X_k,G]$ are $p$-groups for some prime $p$.
\end{lemma}

\begin{proof}
We may apply Lemma \ref{nine} to see that both of the quotient groups $X_{i-1}/X_i$ and $[X_{i-1},G]/[X_i,G]$ are $p_i$-groups for some prime $p_i$ depending on $i$ for all integers $i$ satisfying $j \le i \le k$.  We need to show that all of the $p_i$'s are the same.  It suffices to show that $p_i = p_{i+1}$ for all $i$ such that $j \le i \le k-1$.  If $[X_{i-1},G]$ is not contained in $X_{i+1}$, then $[X_{i-1},G]X_{i+1}/[X_i,G]X_{i+1} = [X_{i-1},G]X_{i+1}/X_{i+1}$ is a nontrivial quotient of $[X_{i-1},G]/[X_i,G]$ and a nontrivial subgroup of $X_i/X_{i+1}$.  Since $[X_{i-1},G]/[X_i,G]$ is a $p_i$-group and $X_i/X_{i+1}$ is a $p_{i+1}$-group, we have $p_i = p_{i+1}$ as desired. Thus, we may assume that $[X_{i-1},G] \le X_{i+1}$.  By Lemma \ref{nine2}, $X_{i-1}/X_{i+1}$ is a $p$-group, and therefore, we must have that $p_i = p_{i+1}$ which proves the result.
\end{proof}

Next, we compile the proof of Theorem \ref{main4}.

\begin{proof}[Proof of Theorem \ref{main4}]
Observe that (1) follows from Lemmas \ref{sixb}, \ref{sixc}, and \ref{nine}.  Conclusion (2) is Lemma \ref{nine2}, Conclusion (3) is Lemma \ref{nine1}, and Conclusion (4) is Lemma \ref{sixd}.
\end{proof}

Let $Z_1 \le Z_2 \le \cdots \le Z_m = Z_{\infty}$ be the terms of the upper central series of $G$.  When $G$ is a $G$ is a nested group with chain of centers $G = X_0 > X_1 > \dots > X_n \ge 1$, we have seen that $Z_1$ is the last term $X_n$.  Notice that $Z_2/Z_1$ will be the last term of the chain of centers for $G/Z_1$.  By Lemma \ref{sixca}, the chain of centers for $G/Z_1$ will be a subchain of $X_0/Z_1, \dots X_{n-1}/Z_1$.  Thus, $Z_2 = X_i$ for some $i \le n-1$.  We will give an example where $i < n-1$.  In a similar fashion, we will see that each $Z_j$ will be some $X_k$.  Note that this means that the chain of centers for $G$ will include a refinement of the upper central series of $G$.  Also, notice that if $Z_\infty < G$, the chain of centers extends beyond the upper central series.   Hence, we may view the chain of centers as an extension of the upper central series of a nested group.

We obtain the following corollaries to Lemmas \ref{nine2} and \ref{nine1}.  In particular, we will show that all but one of the factors of the lower and upper central series of a nested group are elementary abelian $p$-groups for a given prime $p$.

\begin{corollary} \label{nine3}
Suppose $G$ is a nested group.  Then there is a prime $p$ so that $Z_\infty/Z_1$ is a $p$-group and $Z_{i+1}/Z_i$ and $[G,Z_{i+1}]/[G,Z_i]$ are elementary abelian $p$-groups for every integer $i \ge 1$.
\end{corollary}

\begin{proof}
We know that $Z_\infty = Z_l$ for some prime $l$.  If $l = 1$, then the result holds.  Thus, we may assume that $l > 1$.  We know that $Z_l/Z_{l-1} = Z (G/Z_{l-1})$.  By Lemma \ref{sixca}, we have $Z_l = X_{j-1}$ for some integer $j > 1$ and by Lemma \ref{zero}, $Z_1 = X_n$ where $G = X_0 > X_1 > \dots > X_n \ge 1$ is the chain of centers for $G$.  Suppose $i$ is an integer so that $j \le i \le n$.  Then there exists a positive integer $k$ so that $Z_{k-1} \le X_{i} < X_{i-1} \le Z_k$.  We see that $[X_{i-1},G] \le [Z_k,G] \le Z_{k-1} \le X_i$.  Hence, we may apply Lemma \ref{nine2} to obtain $Z_\infty/Z_1 = X_{j-1}/X_n$ is a $p$-group.  Suppose $i$ is an integer, if $Z_{i+1} = Z_i$, then the second conclusion holds.  Thus, we may assume that $Z_i < Z_{i+1}$.  Working as above, we can find positive integers $j \le k$ so that $Z_{i+1} = X_{j-1}$ and $Z_i = X_k$.  Observe that $[X_{j-1},G] = [Z_{i+1},G] \le Z_i = X_k$.  Now, the second conclusion is just Lemma \ref{nine1}.
\end{proof}

We will present examples to see that $Z_1$ need not be a $p$-group or elementary abelian. We now write $G = G_1 \ge G' = G_2 \ge \dots G_l = G_{\infty}$ for the upper central series of $G$.  Note that $G^\infty$ is the nilpotent residual of $G$.  (The smallest normal subgroup of $G$ such that $G/G^{\infty}$ is nilpotent.)

\begin{corollary} \label{nine4}
Suppose $G$ is a nested group.  Then there is a prime $p$ so that $G'/G_\infty$ is a $p$-group and $G_i/G_{i+1}$ are elementary abelian $p$-groups for every integer $i \ge 2$.
\end{corollary}

\begin{proof}
We may assume that $G_\infty = 1$.  This implies that $G$ is nilpotent.  Let $G = X_0 > X_1 > \dots > X_n \ge 1$ be the chain of centers for $G$.  Then $G_2 = [X_0,G]$ and $1 = [X_n,G]$.  By Lemma \ref{sixd}, we have $[X_{i-1},G] \le X_i$ for $i$ such that $1 \le i \le n$.  By Lemma \ref{nine2}, $G_2 = [X_0,G]/[X_n,G]$ is a $p$-group for some prime $p$.  To prove the second part, we may assume that $G_{i+1} = 1$.  We now have that $G_{i-1} \le Z_2$.  This implies that $G_i = [G_{i-1},G] \le [Z_2,G]$.  Also, $[Z_1,G] = 1$.  By Corollary \ref{nine3}, $[Z_2,G] = [Z_2,G]/[Z_1,G]$ is elementary abelian.
\end{proof}

It is easy to see from the definition of nested that if $H$ is a group, $A$ is abelian group, and $G = H \times A$, then $G$ is a nested group if and only if $H$ is a nested group.  As a corollary to Corollary \ref{nine3}, we see that if $G$ is a nested nilpotent group, then $G$ is a direct product of a nested $p$-group and an abelian group.  

\begin{corollary} \label{nine5}
Suppose $G$ is a nilpotent nested group.  Then $G = P \times A$ where $P$ is a nested $p$-group for some prime $p$ and $A$ is an abelian group.
\end{corollary}

\begin{proof}
Since $G$ is nilpotent, we know that $G = Z_\infty$.  By Corollary \ref{nine3}, we see that $G/Z(G)$ is a $p$-group for some prime $p$.  Let $P$ be the Sylow $p$-subgroup of $G$ and let $A$ be the Hall $p$-complement of $G$.  This implies that $A \le Z(G)$ and so $A$ is abelian, and $G = P \times A$.  Since $P \cong G/A$, we see that $P$ is a nested group.
\end{proof}

We conclude this section with an application of Lemmas \ref{sixc} and \ref{sixd}.  In particular, we obtain a characterization of $G/X_1$ when $G$ is a nested group and $G/X_1$ is solvable, but not nilpotent.

\begin{lemma}
Let $G$ be a nested group with chain of centers $G = X_0 > X_1 > \dots > X_n \ge 1$.  If $G/X_1$ is solvable and not nilpotent, then $G/X_1$ is a Frobenius group with Frobenius kernel $(G/X_1)'$.
\end{lemma}

\begin{proof}
By Lemma \ref{sixd}, we know that since $G/X_1$ is not nilpotent, that $G' = [X_0,G] \not\le X_1$.  By Lemma \ref{sixc}, this implies that $G'X_1/X_1 = (G/X_1)'$ is the unique minimal normal subgroup of $G/X_1$.  Thus, $G/X_1$ satisfies Lemma 12.3 of \cite{text}, and since $G/X_1$ is not nilpotent, we see from Lemma 12.3 of \cite{text} that $G/X_1$ must be a Frobenius group with Frobenius kernel $(G/X_1)'$.
\end{proof}

\section{Chains of kernels}

In this section, we consider groups where the kernels of the irreducible characters form a chain.  We write ${\rm Kern} (G) = \{ \ker {\chi} \mid \chi \in \irr G \}$ for the set of all kernels of irreducible characters of $G$.

\begin{lemma} \label{knone}
If $G$ is a group so that ${\rm Kern} (G)$ forms a chain, then $G/G'$ is cyclic.
\end{lemma}

\begin{proof}
If $G' = G$, then the result is trivially true, so we may assume that $G' < G$.  Let $M$ be a maximal subgroup subject to $G' \le M$.  Since $G/G'$ is abelian, this implies that $G/M$ is cyclic.  We can choose $\lambda \in \irr {G/M}$ so that $\lambda$ is faithful.  This implies that $\lambda \in \irr G$ and $M = \ker {\lambda}$.  This implies that $M \in {\rm Kern} (G)$.  Since ${\rm Kern} (G)$ forms a chain, $G/G'$ has a unique maximal subgroup, and thus, $G/G'$ is cyclic.
\end{proof}

Thus, we are able to answer Berkovich's initial question about $p$-groups where ${\rm Kern} (G)$ forms a chain.

\begin{corollary} \label{kntwo}
If $G$ is a nilpotent group and ${\rm Kern} (G)$ forms a chain, then $G$ is a cyclic $p$-group for some prime $p$.
\end{corollary}

\begin{proof}
By Lemma \ref{knone}, we know that $G/G'$ is cyclic.  If $[G',G] < G'$, then $G'/[G',G]$ is a central subgroup of $G/[G',G]$ whose quotient is cyclic.  This implies that $G/[G',G]$ is abelian, and this contradicts $[G',G] < G'$.  We conclude that $G' = [G',G]$, and since $G$ is nilpotent, we conclude that $G' = 1$.  Notice that if $G$ is not a $p$-group, then the kernels of irreducible characters do not form a chain.
\end{proof}

There are many examples of nonnilpotent groups where the kernels form a chain.  The smallest such group would be $S_3$.  Any simple group will also be example.  We now prove Theorem \ref{main1}.


\begin{proof}[Proof of Theorem \ref{main1}]
Suppose ${\rm Kern} (G)$ forms a chain.  This implies that the intersections of kernels of irreducible characters are kernels of irreducible characters.  Hence, every normal subgroup of $G$ is the kernel of an irreducible character, so the normal subgroups form a chain.  Next, suppose that the normal subgroups of $G$ form a chain.  Suppose $G$ has two different chief series.  Then there exist normal subgroups $M$, $N_1$, and $N_2$ in $G$ so that $M$ is in both chief series, $N_1$ and $N_2$ are in different chief series, and $N_1/M$ and $N_2/M$ are minimal normal subgroups of $G/M$.  It follows that $N_1$ and $N_2$ are not comparable, and this contradicts the fact that the normal subgroups of $G$ form a chain.  Finally, we know that every normal subgroup of $G$ lies in some chief series, so if $G$ has a unique chief series, then the subgroups in ${\rm Kern} (G)$ must all lie in that chief series and so they form a chain.
\end{proof}

We now exclude the kernels of the linear characters.  We write 
$${\rm nlKern} (G) = \{ \ker {\chi} \mid \chi \in \nl G \}.$$  We now work to prove Theorem \ref{main2}.  This gives the classification the groups where ${\rm nlKern} (G)$ forms a chain.  One of the key points in our work is that if ${\rm nlKern} (G)$ forms a chain, then $G$ must be a nested group.  This next lemma encodes the key idea to prove this.




\begin{lemma} \label{knthree}
Let $G$ be a group.  If $\gamma, \chi \in \irr G$ satisfy $\ker {\gamma} \le \ker {\chi}$, then $Z(\gamma) \le Z(\chi)$.
\end{lemma}

\begin{proof}
Observe that $[Z(\gamma), G] \le \ker {\gamma} \le \ker {\chi}$.  This implies that 
	$$Z(\gamma) \ker {\chi}/\ker {\chi} \le Z(G/\ker {\chi}) = Z (\chi)/\ker {\chi}.$$
\end{proof}

We show that if ${\rm nlKern} (G)$ is a chain, then $G$ is a nested group.  

\begin{lemma} \label{knfour}
If $G$ is a group so that ${\rm nlKern} (G)$ is a chain, then $G$ is a nested group.
\end{lemma}

\begin{proof}
Suppose $\gamma, \chi \in \irr G$.  If either $\gamma$ or $\chi$ is linear, then its center is $G$, and the center of the other is contained in $G$.  Thus, we may assume that both $\gamma$ and $\chi$ are nonlinear.  Without loss of generality, we have $\ker {\gamma} \le \ker {\chi}$.  By Lemma \ref{knthree}, this implies that $Z(\gamma) \le Z(\chi)$.  We conclude that $G$ is nested.
\end{proof}

Thus, when ${\rm nlKern} (G)$ is a chain, we have the chain of centers $G = X_0 > X_1 > \dots > X_n \ge 1$ by Lemma \ref{defn}.  

Following \cite{waists}, we define a {\it waist} to be a normal subgroup $W$ of a group $G$ with the property that if $N$ is any normal subgroup of $G$, then either $N \le W$ or $W \le N$.  Mann has also studied waists in his preprint \cite{mannpre}.  In particular, Mann notes that if $W$ is a waist in $G$, then $W$ must be the unique normal subgroup of $G$ having order $|W|$.  He also mentions that in $p$-groups, the converse will also be true, but this need not be true otherwise.  We now show that if ${\rm nlKern} (G)$ is a chain, then every normal subgroup of $G$ that is properly contained in $G'$ is a waist for $G$.

\begin{lemma} \label{knfoura}
If $G$ is a nonabelian group so that ${\rm nlKern (G)}$ is a chain and $M < G'$ is normal, then $M$ is a waist for $G$.
\end{lemma}

\begin{proof}
Suppose that there exists a normal subgroup $N$ so that $N \cap M < N$ and $N \cap M < M$.  Let $L = N \cap M$. This implies that there exists such an $N$ where $N/L$ is a chief factor of $G$, so without loss of generality we may assume that $N/L$ is a chief factor for $G$.  We may find $1_M \ne \mu \in \irr {M/L}$.  Also, we can find $1_N \ne \nu \in \irr {N/L}$.  If $N \le G'$, then $NM \le G'$, and note that $NM/L = N/L \times M/L$.  Consider the characters $\chi \in \irr {G \mid 1_N \times \mu}$ and $\psi \in \irr {G \mid \nu \times 1_M}$.  Observe that $G'$ is not in the kernels of $\chi$ and $\psi$; so $\chi$ and $\psi$ are nonlinear.  On the other hand, $N$ is in $\ker {\chi}$ and $M$ is not in $\ker {\chi}$ and $N$ is not in $\ker {\psi}$ and $M$ is in $\ker {\psi}$.  Thus, $\ker {\chi}$ and $\ker {\psi}$ are not comparable which is a contradiction.
	
We must have $N \not\le G'$.  Thus, $G' \cap N < N$.  On the other hand, $L = M \cap N \le G' \cap N$.  Since $N/L$ is a chief factor, we have $G' \cap N = L$.  This implies that $G'N/L = G'/L \times N/L$.  Since $M < G'$, we can find $1_{G'} \ne \alpha \in \irr {G'/M}$.  By inflating $\alpha$, we may view $\alpha$ as a character in $\irr {G'/L}$.  Take $\mu$ and $\nu$ to be the same characters as the last paragraph and fix $\eta \in \irr {G'/L \mid \mu}$.  Consider the characters $\gamma \in \irr {G \mid \alpha \times \nu}$ and $\delta \in \irr {G \mid \eta \times 1_N}$.  Observe that $G'$ is not in the kernels of $\gamma$ and $\delta$; so $\gamma$ and $\delta$ are nonlinear.  On the other hand, $M$ is in $\ker {\gamma}$ and $N$ is not in $\ker {\gamma}$ and $M$ is not in $\ker {\delta}$ and $N$ is in $\ker {\delta}$.  Thus, $\ker {\gamma}$ and $\ker {\delta}$ are not comparable.  This is a contradiction.  Thus, no such $N$ can exist, and $M$ is a waist for $G$.
\end{proof}

We now obtain a result on the chain of centers for these groups.

\begin{lemma} \label{knfive}
If $G$ is a nonabelian group so that ${\rm nlKern (G)}$ is a chain and we write $G = X_0 > X_1 > \dots X_n > 1$ for the chain of centers, then $W_i/[X_i,G]$ is cyclic of prime power order where $W_i = X_i \cap G'$ for each $i$ satisfying $i = 1, \dots, n$.
\end{lemma}

\begin{proof}
Observe that we may assume that $[X_i,G] = 1$.  If $W_i = 1$, then the result holds.  Suppose $W_i > 1$.  Notice that any subgroup of $W_i$ will be normal in $G$, and by Lemma \ref{knfoura}, will be a waist.  Thus, $W_i$ has a unique maximal subgroup and this implies that $W_i$ is cyclic of prime power order.
\end{proof}

This next result encapsulates a portion of the work to prove Theorem \ref{main2}.

\begin{lemma} \label{knfive1}
If $G$ is a nonabelian group so that ${\rm nlKern (G)}$ is a chain and write $G = X_0 > X_1 > \dots X_n > 1$ for the chain of centers, then one of the following occurs:
\begin{enumerate}
\item $G' < X_1$, $[G',G] = [X_1,G] > X_2$, $|G':[G',G]|= p$, and $X_1/[G',G] = P/[G',G] \times Q/[G',G]$ such that $P/[G',G]$ and $Q/[G',G]$ are each cyclic of prime power order and have coprime orders.
\item $G'/(X_1 \cap G')$ is a chief factor of $G$ and $X_1/(X_1 \cap G')$ is cyclic of prime power order.
\end{enumerate}
\end{lemma}

\begin{proof}
Suppose first that $G' < X_1$.  We see that $[G',G] \le [X_1,G] < [X_0,G] = G'$ by Corollary \ref{threec}.    In light of Corollary \ref{nine4}, we know that $G'/[G',G]$ is an elementary abelian $p$-group for some prime $p$.  For the moment, we may assume that $[G',G] = 1$.  In the notation of Lemma \ref{knfive}, we have $W_1 = X_1 \cap G' = G'$, and by that lemma, we know that $G' = W_1 = W_1/[X_1,G]$ is cyclic.  Now, $G'$ is both elementary abelian and cyclic.  This implies that $|G'| = p$.
	
Suppose that there exists a subgroup $Y$ of order $p$ in $X_1$ that is not $G'$. Then $Y$ is central in $G$.  Let $\lambda \in \irr {G'}$ be nonprincipal, and let $\nu \in \irr Y$ be nonprincipal.  Consider $\chi \in \irr {G \mid 1_Y \times \lambda}$ and $\psi \in \irr {G \mid \nu \times \lambda}$.  Since $\lambda \ne 1_{G'}$, we conclude that $\chi$ and $\psi$ are not linear.  Observe that $\ker{\chi} \cap YG' = \ker{1_Y \times \lambda} = Y$ and $\ker {\psi} \cap YG' = \ker {\nu \times \lambda}$.  Now, $\nu \times \lambda$ is a nonprincipal irreducible character of $YG'$, so its kernel must have order $p$ since $YG'$ is an elementary abelian group of order $p^2$.  Also, $\ker {\nu \times \lambda}$ is neither $Y$ nor $G'$.  It follows that $\ker{\nu \times \lambda}$ is not comparable with $Y$.  This implies that $\ker {\chi} \cap YG'$ is not comparable to $\ker {\psi} \cap YG'$.  We deduce that $\ker {\chi}$ and $\ker {\psi}$ are not comparable, and this contradicts the fact that ${\rm nlKern} (G)$ is a chain.  Therefore, $G'$ is the only subgroup of order $p$ in $X_1$, and so, if $P$ is the Sylow $p$-subgroup of $X_1$, then $P$ is cyclic.
	
Suppose now that $Y_1$ and $Y_2$ are distinct minimal normal subgroups of $G$ whose orders are not $p$.  Let $1 \ne \nu_i \in \irr {Y_i}$.  Take $\lambda$ as in the previous paragraph.  Consider $\gamma_1 \in \irr {G \mid \nu_1 \times 1_{Y_2} \times \lambda}$ and $\gamma_2 \in \irr {G \mid 1_{Y_1} \times \nu_2 \times \lambda}$.  Since $\lambda \ne 1_{G'}$, we conclude that both $\gamma_i$ are not linear.  Also, $Y_1$ is not contained in $\ker {\gamma_1}$ and $Y_2$ is contained in $\ker {\gamma_1}$ while $Y_1$ is contained in $\ker {\gamma_2}$ and $Y_2$ is not contained in $\ker{\gamma_2}$.  It follows that $\ker {\gamma_1}$ and $\ker {\gamma_2}$ are not comparable, and this contradicts the fact that ${\rm nlKern} (G)$ is a chain.  We deduce that $X_1$ contains at most one minimal normal subgroup of $G$ whose order is not $p$.  If $Q$ is a Hall $p$-complement of $X_1$, then $X_1 = P \times Q$ and $Q$ is cyclic of prime power order for some prime other than $p$.  Notice that this implies that $X_1$ is cyclic.
	
Removing the assumption that $[G',G] = 1$, we obtain $|G':[G',G]| = p$.  Since $[G',G] \le [X_1,G] < G'$, this implies that $[G',G] = [X_1,G]$.  To prove conclusion (2), we need to show that $[G',G] > X_2$.  By Lemma \ref{knfoura} we know that $[G',G]$ is a waist for $G$, so either $[G',G] \le X_2$ or $X_2 < [G',G]$.
	
We now suppose that $[G',G] \le X_2$.  If $G' \le X_2$, then $[G',G] \le [X_2,G]$, but by Corollary \ref{threec}, we have $[X_2,G] < [X_1,G] = [G',G]$, which is a contradiction.  Therefore, we have $G' \not\le X_2$.  By Lemma \ref{sixd}, since $[G',G] = [X_1,G]$, this implies that $G/[X_2,G]$ is a nilpotent group.  Applying Lemma \ref{nine1}, we know that $G/X_2 = X_0/X_2$ and $G'/[X_2,G] = [X_0,G]/[X_2,G]$ are $p$-groups.  Using Lemma \ref{nine}, $X_1/X_2$ is an elementary abelian $p$-group.  On the other hand, we have seen that $X_1/[G',G]$ is cyclic.  We conclude that $|X_1:X_2| = p$.  Since $G'$ is not contained in $X_2$, we see that $X_1 = X_2 G'$.  Notice that $G'/[G',G]$ is the unique subgroup of order $p$ in $X_1/[G',G]$ and as $X_2$ does not contain $G'$, we conclude that $X_2/[G',G]$ is a $p'$-group.  Using the result of the previous paragraph, we see that $X_2/[G',G]$ is a $q$-group for some prime $q$.  We may now deduce that $G'/[X_2,G]$ is the Sylow $p$-subgroup of $X_1/[X_2,G]$.  Let $Q/[X_2,G]$ be the Sylow $q$-subgroup of $X_2/[X_2,G]$.  We obtain $X_1/[X_2,G] = Q/[X_2,G] \times G'/[X_2,G]$.  If $Q = [X_2,G]$, then $X_1 = G'$ which violates our assumption that $G' < X_1$.  Thus, we may assume that $[X_2,G] < Q$.  In particular, we may find $1_Q \ne \lambda \in \irr {Q/[X_2,G]}$.  Since $|G':[G',G]| = p$, we can find $\nu \in \irr {G'/[G',G]}$ so that $\ker {\nu} = [G',G]$.  We will view $\mu$ as a character in $\irr {G'/[X_2,G]}$.  Since $[X_2,G] < [G',G]$, we can find $1_{[G',G]} \ne \eta \in \irr {[G',G]/[X_2,G]}$ and take $\mu \in \irr {G'/[X_2,G] \mid \eta}$.  Consider the characters $\chi \in \irr {G \mid 1_Q \times \mu}$ and $\psi \in \irr {G \mid \lambda \times \nu}$.  Observe that $G'$ is not in the kernels of $\chi$ and $\psi$; so $\chi$ and $\psi$ are nonlinear.  On the other hand, $Q$ is in $\ker {\chi}$ and $[G',G]$ is not in $\ker {\chi}$ and $Q$ is not in $\ker {\psi}$ and $[G',G]$ is in $\ker {\psi}$.  Thus, $\ker {\chi}$ and $\ker {\psi}$ are not comparable.  This is a contradiction.  Hence, we have $X_2 < [G',G]$.

We now assume that $G' \not\le X_1$.  This implies that $X_1 \cap G' < G'$.  Thus, we may assume that $X_1 \cap G' = 1$.  It follows that $[X_1,G] \le X_1 \cap G'$, so $[X_1,G] = 1$.  Recall that $[X_0,G] = G'$.  Since $G' \not\le X_1$, we have that $X_1 G'$ is the unique minimal normal subgroup of $G/X_1$ by Lemma \ref{sixc}.  Notice that $X_1 G' = X_1 \times G'$, and this implies that $G'$ is a minimal normal subgroup of $G$.
	
We now show that $X_1$ is cyclic of prime order.  If $X_1 = 1$, then the result holds.  Suppose $X_1 > 1$ and there exist distinct subgroups $M_1$ and $M_2$ so that $M_j$ is maximal in $X_1$ for $j = 1,2$.  Notice that $X_1$ is central in $G$, so $M_j$ is normal in $G$ and $X_1/M_j$ has prime order.  Thus, we can find characters $1 \ne \lambda_j \in \irr {X_1/M_j}$ so that $\ker{\lambda_j} = M_j$.  Fix $1 \ne \alpha \in \irr {G'}$. Let $\chi_j \in \irr {G \mid \lambda_j \times \alpha}$.  Note that $\alpha$ is a constituent of $(\chi_j)_{G'}$, so $G' \not\le \ker {\chi_j}$, and hence, we have $\chi_j (1) \ne 1$.  Observe that $M_1 = \ker {\chi_1} \cap X_1$ is not comparable to $M_2 = \ker {\chi_2} \cap X_1$, and this implies that $\ker {\chi_1}$ is not comparable to $\ker {\chi_2}$, a contradiction.  Thus, $X_1$ has a unique maximal subgroup and this implies that $X_1$ is cyclic of prime power order.
	
\end{proof}

The following theorem includes Theorem \ref{main2}.

\begin{theorem}
Let $G$ be a group.  Then the following are equivalent:	
\begin{enumerate}
\item ${\rm nlKern} (G)$ forms a chain.
\item $G$ satisfies both:
\begin{enumerate}
	\item Every normal subgroup $M$ of $G$ with $1 \le M < G'$ is a waist.
	\item There is a unique subgroup $M$ so that $G'/M$ is a chief factor for $G$.  The quotient $G/M$ has at most two minimal normal subgroups.
\end{enumerate} 
\item Every normal subgroup $M$ of $G$ with $1 \le M < G'$ satisfies either: 
\begin{enumerate} 
	\item $G'/M$ is not a chief factor of $G$ and $G/M$ has a unique minimal normal subgroup, or 
	\item $G'/M$ is a chief factor of $G$ and $G/M$ has at most two minimal normal subgroups.
\end{enumerate}
\end{enumerate}
\end{theorem}

\begin{proof}
We first show that (1) and (2) are equivalent.  Suppose Condition (1).  By Lemma \ref{knfoura}, every normal subgroup $M$ of $G$ with $1 \le M < G'$ is a waist.  Let $M$ be a subgroup of $G'$ that is normal in $G$ so that $G'/M$ is a chief factor of $G$.  Since $M$ is a waist, it is the unique normal subgroup of $G$ that is contained in $G'$ such that $G'/M$ is a chief factor.  Observe that $G'/M$ is a minimal normal subgroup of $G/M$.  Suppose $N/M$ is some other minimal normal subgroup of $G/M$.  It follows that $N \cap G' = M$ and so, $N/M \le Z(G/M)$.  We now apply Lemma \ref{knfive1}.  From that lemma, we have that $M = [G',G]$ when $G' < X_1$ and $M = X_1 \cap G'$ otherwise.
	
In the case when $G' < X_1$, we see that $X_1/[G',G] = P/[G',G] \times Q/[G',G]$ where $P/[G',G]$ and $Q/[G',G]$ are cyclic groups of coprime prime power orders.  We know that $[G',G] = [X_1,G]$ and $X_1/[X_1,G] = Z(G/[X_1,G])$.   Note that $G' \le X_1$, and so, it follows that the only possible minimal normal subgroups of $G/M$ are the socles of $P/M$ and $Q/M$.  Therefore, $G/M$ has at most two minimal normal subgroups in this case.
	
Now, suppose that $G' \not< X_1$.  We have that $M = G' \cap X_1$.  Observe that $[X_1,G] \le X_1 \cap G' = M$, and so, $X_1/M \le Z(G/M)$.  Since $G/M$ is not abelian, we have $Z(G/M) < X_0/M$, and thus, $Z(G/M) = X_1/M$.  We know that $X_1/M$ is cyclic of prime power order.  Thus, the only possible minimal normal subgroups of $G/M$ are $G'/M$ and the socle of $X_1/M$.  We conclude that $G/M$ has at most two minimal normal subgroups of $G$.  This proves Condition (2).
	
Suppose Condition (2).  Hence, every normal subgroup $M$ of $G$ with $1 \le M < G'$ is a waist and when $M$ satisfies $G'/M$ is a chief factor for $G$, the quotient $G/M$ has at most two minimal normal subgroups.  Let $\chi$ and $\psi$ are nonlinear irreducible characters of $G$.  If either $\ker {\chi}$ or $\ker {\psi}$ is a proper subgroup of $G'$, then it will be a waist, and we must have that $\ker {\chi} \le \ker {\psi}$ or $\ker{\psi} \le \ker {\chi}$ as desired.  Thus, we have that neither $\ker {\chi}$ nor $\ker{\psi}$ is contained in $G'$.
	
Let $M$ be a normal subgroup of $G$ contained in $G'$ so that $G'/M$ is a chief factor for $G$.  We know that $M$ is a waist, and since $M$ is a waist, it will be the unique normal subgroup of $G$ contained in $G'$ whose quotient is a chief factor for $G$.  Observe that $G'/M$ is a minimal normal subgroup of $G/M$.
	
Clearly, neither $\ker {\chi}$ nor $\ker {\psi}$ is contained in $M$, and since $M$ is waist, we must have $M \le \ker {\chi}$ and $M \le \ker {\psi}$.  Since $\chi$ and $\psi$ are not linear, we must have $\ker {\chi} \cap G' < G'$ and $\ker {\psi} \cap G' < G'$.   We deduce that $\ker {\chi} \cap G' = M = \ker {\psi} \cap G'$.  It follows that $\ker {\chi}/M \le Z(G/M)$ and $\ker {\psi}/M \le Z(G/M)$.
	
There are two possibilities: either $G'/M \le Z(G/M)$ or $G'/M \not\le Z(G/M)$.  For the first possibility, take $p$ to be the prime that divides $|G':M|$, and let $P/M$ be the Sylow $p$-subgroup of $Z(G/M)$.  If $P/M$ is not cyclic, then it will contain at least $p+1$ minimal normal subgroups which are necessarily minimal normal in $G/M$ which is a contradiction.  Let $Q/M$ be the Hall $p$-complement of $Z(G/M)$.  Since $Q/M$ contains at most one minimal normal subgroup, we deduce that $Q/M$ is cyclic of prime power order.  Since $\ker {\chi}/M$ and $\ker {\psi}/M$ are contained in $Z(G/M)$ and do not contain $G'/M$, we conclude that $\ker {\chi}$ and $\ker {\psi}$ are contained in $Q$.  Since $Q/M$ is cyclic of prime power order, we have $\ker {\chi} \le \ker{\psi}$ or $\ker {\psi} \le \ker {\chi}$ as desired in this case.
	
Now, suppose that $G'/M \not\le Z(G/M)$.  Now, $Z(G/M)$ contains at most one minimal normal subgroup, and so, $Z(G/M)$ must be cyclic of prime power order.  Since $\ker {\chi}/M$ and $\ker {\psi}/M$ are contained in $Z(G/M)$, this implies that $\ker {\chi} \le \ker{\psi}$ or $\ker {\psi} \le \ker {\chi}$, and we have Condition (1).
	
We now show that (2) and (3) are equivalent.  Suppose (2).  Let $M < G'$ be normal in $G$.  If $G'/M$ is a chief factor for $G$, we know that $G/M$ has at most two minimal normal subgroups of $G$.  We suppose that $G'/M$ is not a chief factor for $G$.  We can find $M < N < G'$ so that $N/M$ is minimal normal in $G/M$.  By (2), we know $N$ is a waist for $G$.  If $K/M$ is minimal normal in $G$, then either $K \le N$ or $N \le K$.  Since $N/M$ is minimal normal, the first case yields $K = N$, and since $K/M$ is minimal normal, the second case yields the same thing.  Therefore, $G/M$ has a unique minimal normal subgroup.  This proves (3).
	
Suppose (3).  Let $M < G'$ be normal in $G$.  If $G'/M$ is a chief factor for $G$, we know that $G/M$ has at most two minimal normal subgroups of $G$.  We suppose that $G'/M$ is not a chief factor for $G$.  We need to show that $M$ is a waist for $G$.  Let $N$ be a normal subgroup of $G$.  If $M \cap N = M$, then $M \le N$.  Thus, we suppose that $M \cap N < M$.  We know that $G/(M \cap N)$ contains a unique minimal normal subgroup $L/(M \cap N)$.  Since $M \cap N < M$, it follows that $L \le M$.  If $M \cap N < N$, then $L \le N$.  This implies that $M \cap N < L \le M \cap N$ which is a contradiction.  Thus, we must have $M \cap N = N$ and so, $N \le M$.  This implies (2).
\end{proof}

\section{Nested by degrees} \label{degs}

We say that $G$ is a group that is {\it nested by degrees} if all characters $\chi, \psi \in \irr G$ satisfying $\psi (1) \le \chi (1)$ also satisfy $Z (\chi) \le Z (\psi)$.   Obviously, in this case, $G$ is a nested group, so $G$ will have the chain of centers $G = X_0 > X_1 > \dots > X_n \ge 1$.  If $\chi, \nu \in \irr G$ with $\chi (1) = \nu (1)$, then $\chi (1) \le \nu (1)$ implies $Z (\nu) \le Z (\chi)$ and $\nu (1) \le \chi (1)$ implies $Z (\chi) \le Z(\nu)$.  Thus, irreducible characters with the same degrees have the same centers.

Let $1 = d_0 < d_1 < \cdots < d_r$ be the distinct irreducible character degrees for $G$.  Define $Y_j = Z (\chi)$ when $\chi \in \irr G$ satisfies $\chi (1) = d_j$.  We know that there is a positive integer $i$ so that $Y_j = X_i$, and by Lemma \ref{twoa}, $[Y_j,G] = [X_i,G] \le \ker{\chi}$.  We say that $Y_j$ is the $j$th degree center of $G$. Observe that $G = Y_0 \ge Y_1 \ge \cdots \ge Y_r = Z (G)$.  

\begin{lemma}
Let $G$ be nested by degrees, let $Y_j = Y_j (G)$ be defined as above with $0 \le j \le r$ where $r = |\cd G|$, and let $G = X_0 > X_1 > \dots > X_n \ge 1$ be the chain of centers for $G$.  Then the following hold:
\begin{enumerate}
	\item $r \ge n$,
	\item $X_i \le Y_i$ for $i = 1, \dots n$,
	\item for each integer $j \in \{ 1, \dots, r \}$, there exists a unique integer $i \le j$ so that $Y_j = X_i$,
	\item $Y_1 = X_1 < G$,
	\item for all integer $j \in \{ 1, \dots r\}$, we have $Y_{j+1} < Y_j$ if and only if $[Y_{j+1},G] < [Y_j,G]$.
\end{enumerate} 
\end{lemma}

\begin{proof}
Notice that the distinct $Y_j$'s will form a chain of centers in the sense of Lemma \ref{defn}.   The $Y_j$'s need not be distinct so we do not necessarily have $X_j = Y_j$.   It is not difficult to see that, $X_i \le Y_i$ for $i = 0, 1, \dots, n$, and for each $j = 1, \dots r$, we have $Y_j = X_i$ for some $i \le j$.  Since the $X_i$'s are distinct, it follows that $n \le r$.  Note that $Y_1 = X_1 < G$.  Applying Lemma \ref{threea}, we see that $Y_j/[Y_j,G] = X_i/[X_i,G] = Z(G/[X_i,G]) = Z(G/[Y_j,G])$.  Observe that if $Y_j < Y_{j+1}$, then $Y_{j+1} = X_{i+1}$; so $[Y_j,G] < [Y_{j+1},G]$ by Corollary \ref{threec}.  Conversely, if $Y_{j+1}= Y_j$, then $[Y_j,G] = [Y_{j+1},G].$  We deduce that $Y_{j+1} < Y_j$ if and only if $[Y_{j+1},G] < [Y_j,G]$.
\end{proof}

When $N$ is a normal subgroup of $G$, we write $\irr {G \mid N}$ for the characters in $\irr G$ whose kernels do not contain $N$ and $\cd {G \mid N} = \{ \chi (1) \mid \chi \in \irr {G \mid N}\}$.  We now assume that $G$ is a nested group where $Y_i < Y_{i-1}$ for some or all $3 \le i \le r$.

\begin{lemma} \label{seven}
Suppose $G$ is nested by degrees with $G= Y_0 \ge Y_1 \ge \cdots \ge Y_r$ the degree centers.
\begin{enumerate}
\item If $i < j$ satisfies $Y_i > Y_j$ and $\nu \in \irr G$ satisfies $\nu (1) = d_j$, then $[Y_i,G] \not\le \ker {\nu}$.
\item If $Y_{i+1} < Y_i$, then $\cd {G/[Y_i,G]} = \{ 1, d_1, \dots, d_i \}$.
\item If $i < j$ satisfies $Y_i > Y_{i+1} = Y_j > Y_{j+1}$, then 
	$$\cd {G/[G,Y_j] \mid [G,Y_i]/[G,Y_j]} = \{ d_{i+1}, \dots, d_j \}.$$
\end{enumerate}
\end{lemma}

\begin{proof}
Suppose that $Y_i > Y_j$.  If $[Y_i,G] \le \ker {\nu}$, then Lemma \ref{twoa} would imply that $Y_i \le Z(\nu) = Y_j$, a contradiction.  This proves (1).  Suppose $Y_{i+1} < Y_i$.  Fix a integer $k \le i$, and consider a character $\nu \in \irr G$ with $\nu (1) = d_k$.  We saw above that $\ker {\nu} \ge [Y_k,G] \ge [Y_i,G]$, and so, $\nu \in \irr {G/[Y_i,G]}$ and $d_j \in \cd {G/[Y_i,G]}$.  If $k > i$, then (1) implies that $[Y_i,G]$ is not contained in $\ker {\nu}$, and so, $d_k \not\in \cd {G/[Y_i,G]}$.  This proves (2).  Now, suppose that $Y_i > Y_{i+1} = Y_j > Y_{j+1}$. By (2), $\cd {G/[Y_i,G]} = \{ 1, \dots, d_i \}$ and $\cd {G/[Y_j,G]} = \{ 1, \dots, d_j \}$.  It follows that $\{ d_{i+1}, \dots, d_j \} \subseteq \cd {G/[Y_j,G] \mid [Y_i,G]/[Y_j,G]}$.  Suppose $\chi \in \irr {G/[Y_j,G] \mid [Y_i,G]/[Y_j,G]}$.  Since $[Y_i,G] \not\le \ker{\chi}$, we know that $Y_i \not\le Z(\chi)$, and so, $\chi (1) > d_i$.  Conclusion (3) now follows.  
\end{proof}

Applying Lemma \ref{sixd} when $G$ is nilpotent and nested by degrees, we have that $Y_i/Y_{i+1} \le Z (G/Y_{i+1})$ for all $i$ with $0 \le i \le r-1$.  (This is was previously proved by Nenciu in Proposition 2.4 of \cite{nested} for $p$-groups that are nested by degrees).  We will now show that this is an equality when $Y_i < Y_{i+1}$.

\begin{lemma} \label{four}
Suppose $G$ is nested by degrees with $G= Y_0 \ge Y_1 \ge \cdots \ge Y_r$ the degree centers and $1 = d_0 < d_1 < \cdots < d_r$ the distinct irreducible character degrees, $N$ is a normal subgroup of $G$, and $i$ is the maximum index so that $d_i \in \cd {G/N}$.  Then $Z(G/N) = Y_i/N$.  In particular, if $Y_{i+1} < Y_i$, then $Y_i/[Y_i,G] = Z (G/[Y_i,G])$.
\end{lemma}

\begin{proof}
Let $Z/N = Z (G/N)$.  Observe that $Z = \cap_{\theta \in \irr {G/N}} Z (\theta)$.  Let $\theta \in \irr {G/N}$, and observe that $\theta (1) \le d_i$.  Thus, $Y_i \le Z (\theta)$.  Since this is true for all $\theta \in \irr {G/N}$, we have $Y_i \le \cap_{\theta \in \irr {G/N}} Z(\theta) = Z$.  Suppose now that $\theta (1) = d_i$, and note that $Z \le Z (\theta) = Y_i$.  We conclude that $Z = Y_i$.  When $Y_{i+1} < Y_i$, we can apply Lemma \ref{seven} to see that $d_i$ is the maximal degree in $\cd {G/[Y_i,G]}$, and the second conclusion now follows from the first.
\end{proof}

We say that a group $G$ is {\it strictly nested by degrees} if $G$ is nested by degrees and all characters $\chi, \psi \in \irr G$ satisfying $\psi (1) < \chi (1)$ have $Z (\chi) < Z (\psi)$.  The following shows that the $X_i$'s and the $Y_i$'s are identical when $G$ is strictly nested by degrees.

\begin{lemma} \label{fivea}
Let $G$ be nested by degrees with the chain of centers $G = X_0 > X_1 > \dots > X_n \ge 1$ and $G= Y_0 \ge Y_1 \ge \cdots \ge Y_r$ the degree centers.  Then the following are equivalent:
\begin{enumerate}
\item $G$ is strictly nested by degrees,
\item $Y_i > Y_{i+1}$ for $i = 0, \dots, r - 1$,
\item $r = n$ and $X_i = Y_i$ for $i = 0, \dots, r$.
\end{enumerate}
\end{lemma}

\begin{proof}
Suppose $G$ is strictly nested by degrees.  If $\chi, \psi \in \irr G$ satisfy $\chi (1) = d_i$ and $\psi (1) = d_{i+1}$, then $\chi (1) < \psi (1)$, and so, $\ker {\chi} > \ker {\psi}$.  We have $\ker {\chi} = Y_i$ and $\ker {\psi} = Y_{i+1}$.  This implies that $Y_i > Y_{i+1}$.  This proves (1) implies (2).  Note that if (2) holds, then $G = Y_0 > Y_1 > \dots > Y_r \ge 1$ is the chain of centers of the characters of $G$.  By the definition of the $X_i$'s, we have $r = n$ and $X_i = Y_i$ for $i = 0, \dots, r$ which is (3). Finally, if (3) holds and $\chi, \psi \in \irr G$ satisfy $\chi (1) = d_i$ and $\psi (1) = d_j$ with $j > i$, then $Z (\chi) = Y_i = X_i > X_j = Y_j = Z (\psi)$, and we conclude that (1) holds.
\end{proof}

\section{Nested GVZ-groups}

We now turn to nested GVZ-groups.  In Proposition 1.2 in \cite{gvz}, it was proved that if $G$ is a GVZ-group, then $G$ is nilpotent.  In light of Corollary \ref{nine5}, we know that nilpotent nested groups are the direct product of a nested $p$-group for some prime $p$ with an abelian group.  Therefore, we may assume that nested GVZ-groups are $p$-groups for some prime $p$.  

\begin{lemma} \label{strict}
Let $G$ be a GVZ group.  Then the following are equivalent:
\begin{enumerate}
\item $G$ is a nested group,
\item $G$ is nested by degrees,
\item $G$ is strictly nested by degrees.
\end{enumerate}
Suppose now that $G$ is a nested GVZ group and write $G = X_0 > X_1 > \dots > X_n \ge 1$ for the chain of centers of $G$.  If $1 = d_0 < d_ 1 < \dots < d_r$ are the degrees in $\cd G$, then $r = n$ and $d_i^2 = |G:X_i|$ for all integers $0 \le i \le r$.
\end{lemma}

\begin{proof}
Observe that if $G$ is strictly nested by degrees, then $G$ is nested.  Thus, we have that (3) implies (1).  Suppose $G$ is a GVZ group and is nested.  For each $X_i$, there is a character $\chi \in \irr G$ so that $Z (\chi) = X_i$.  Also, we know that for every $\psi \in \irr G$ that there is an $i$ so that $Z(\psi) = X_i$.  Since $G$ is a GVZ-group, we know that $\psi$ vanishes on $G \setminus Z(\psi)$.  Applying Corollary 2.30 of \cite{text}, this implies that $\psi (1)^2 = |G: Z (\psi)| = |G:X_i|$.  We now have $\cd G = \{ |G:X_i|^{1/2} \mid i = 0, \dots, n \}$.  Since $1 = d_0 < d_1 < \dots < d_r$ are the distinct degrees in $\cd G$, we obtain $r = n$ and matching up the corresponding values, we obtain $|G:X_i| = d_i^2$.  This yields the conclusion about the degrees.

Suppose that $\chi, \psi \in \irr G$ where $\chi (1) = d_i$ and $\psi = d_j$.  If $\chi (1) \ge \psi (1)$, then $i \ge j$ and $Z (\chi) = X_i \le X_j = Z(\psi)$.  This implies that $G$ is nested by degrees.  This yields that (1) implies (2).  If $G$ is nested by degrees and we have characters $\chi$ and $\psi$ as above with $\chi (1) > \psi (1)$, then since $d_i^2 = |G:X_i|$ and $d_j^2 = |G:X_j|$, we have $Z(\chi) = X_i < X_j = Z(\psi)$.  We conclude that $G$ is strictly nested by degrees which proves (2) implies (3).
\end{proof}

Observe that if $G$ is a nested group, then since $G = X_0$, we have $G' = [X_0,G]$, and so if $G$ is a nested GVZ group, then $G' = [X_0,G] \le X_1$ by Lemma \ref{sixd}.  We now show that if $G$ is a nested GVZ-group and $G' = X_1$, then $G$ is in fact a semi-extraspecial group.

\begin{lemma}
If $G$ is a nested GVZ-group and $G = X_0 > X_1 > \dots > X_n \ge 1$ is the chain of centers of $G$, then either $G$ is a semi-extraspecial $p$-group for some prime $p$ or $G' < X_1$.
\end{lemma}

\begin{proof}
Since $G$ is a GVZ-group, every nonlinear irreducible character vanishes off of its center, and since $G$ is nested, the center of every nonlinear irreducible character is contained in $X_1$.  This implies that every nonlinear irreducible character of $G$ vanishes off of $X_1$.  If $X_1 = G'$, then this implies that $G$ is a Camina $p$-group by Proposition 3.1 of \cite{ChMc}.
\end{proof}

Any nilpotent group with nilpotence class $2$ will be a GVZ-group by Theorem 2.31 of \cite{text}.  Thus, any nested group with nilpotence class $2$ will be a nested GVZ-group.  We now show for $p$-groups that are nested by degrees that $Y_2 = X_2$.


\begin{lemma} \label{six}
If $G$ is a $p$-group that is nested by degrees with the chain of centers $G = X_0 > X_1 > \dots > X_n \ge 1$ and $G= Y_0 \ge Y_1 \ge \cdots \ge Y_r$ the degree centers where $r \ge 2$, then $Y_2 = X_2 < Y_1$.
\end{lemma}

\begin{proof}
Suppose $Y_2 \ne X_2$.  This implies that $Y_2 = X_1 = Y_1$.  Then $G/[Y_2,G] = G/[Y_1,G]$, and since $G' \le Y_1$, we see that $G/[Y_1,G]$ has nilpotence class $2$.  Thus, $G/[Y_1,G]$ is a nested GVZ-group, and so, in view of Lemma \ref{strict}, $G/[Y_1,G]$ is strictly nested by degrees.  As we saw earlier, we may apply Lemma \ref{twoa} to see that $d_1, d_2 \in \cd {G/[Y_2,G]}$.  We deduce from Lemma \ref{fivea} that $Y_2 < Y_1$ which is a contradiction.
\end{proof}

We note that if $G$ is not nilpotent, then Lemma \ref{six} is not necessarily true.  Let $G = {\rm SL}_2 (3)$.  It is not difficult to see that $G$ is nested by degrees.  We know that $\cd G = \{ 1, 2, 3 \}$ and $Y_1 = Y_2 = X_1 = Z(G)$ and $n = 1$ so we do not have $Y_2$ equal to $X_2$.

When $G$ is a $p$-group that nested by degrees but is not a GVZ-group, it is not clear that we must have $Y_{i+1} < Y_i$ for $i \ge 3$.  However, we have seen that nested GVZ-groups do satisfy this condition.

We can apply our results about nested GVZ-groups to obtain results about nonabelian nested $p$-groups.  We say a group $G$ is a {\it VZ-group} if every irreducible character of $G$ vanishes on $G \setminus Z(G)$.  We first used the term VZ-groups in \cite{VZ} although these groups had been studied earlier in \cite{FAMo} by Fern\'andez-Alcobar and Moret\'o.  In particular, Fern\'andez-Alcobar and Moret\'o have proved that $G$ is a VZ-group if and only if $G$ is isoclinic to a semi-extraspecial group.  We note that Nenciu's definition of GVZ-groups in \cite{gvz} was a generalization of VZ-groups.  Observe that this next theorem includes conclusions (1), (2), and (3) of Theorem \ref{main5}.

\begin{theorem}\label{sevena}
If $G$ is a nonabelian nested $p$-group with chain of centers $G = X_0 > X_1 > \dots X_n > 1$ and $d_i^2 = |G:X_i|$ for $i = 0, \dots n$, then
\begin{enumerate}
\item $|G:X_1|$ is a square and $\cd {G/[X_1,G] \mid G'/[X_1,G]} = \{ d_1 \}$.
\item If $G' \le X_i$ for some $i = 1, \dots, n$, then $|G:X_j|$ is a square and $\cd {G/[X_j,G] \mid [X_{j-1},G]/[X_j,G]} = \{ d_j \}$ for all $1 \le j \le i$.
\item If the subgroup $N$ satisfies $[X_1,G] \le N < NG'$, then $G/N$ is a $VZ$-group with center $X_1/N$.  Furthermore, if $NG' = X_1$, then $G/N$ is a semi-extraspecial group.
\item $G/[X_1,G]$ is a VZ-group, so $|G':[X_1,G]| \le d_1$.
\end{enumerate}
\end{theorem}

\begin{proof}
By Lemma \ref{threea}, $X_i/[X_i,G] = Z(G/[X_i,G])$.  Thus, $G' \le X_i$ implies that $G/[X_i,G]$ has nilpotence class $2$, so $G/[X_i,G]$ is a GVZ-group.  By Lemma \ref{strict}, $|G:X_j|$ is a square and $d_j = |G:X_j|^{1/2}$ for $1 \le j \le i$, and $G/[X_i,G]$ is strictly nested by degrees which gives the first part of Conclusion (2).  The last part of Conclusion (2) follows from Lemma \ref{seven} (4).  Conclusion (1) follows from Conclusion (2) with the observation that $G' = [X_0,G] \le X_1$.  To prove Conclusion (3), observe that $G' \not\le N$ and $[X_1,G]$, so $X_1/N = Z(G/N)$ by Lemma \ref{sixa}.  Since $G/N$ is a GVZ-group, every nonlinear irreducible character of $G/N$ vanishes off of $X_1/N$.  By definition of VZ-group, since $X_1/N = Z(G/N)$ this implies that $G/N$ is a VZ-group.  If $X_1 = NG'$, then $Z(G/N) = (G/N)'$, and it is known that a VZ-group with center and derived subgroup equal must be semi-extraspecial.  The first part of Conclusion (4) follows from Conclusion (3) since $[X_1,G] < [X_0,G] = G'$ by Corollary \ref{threec}. The second part of Conclusion (4) follows from Lemma 2.4 of \cite{GCG}.
\end{proof}

We now show that when $G$ is a nested $p$-group with nilpotence class $2$, then we can bound the size of the derived subgroup of $G$ in terms of the indices $|G:X_i|$ for $i = 1, \dots, n$.  This should be thought of as a generalization of the result for semi-extraspecial groups and VZ-groups that the size of the derived subgroup is at most the square root of the index of the center.

\begin{theorem} \label{bound}
Let $G$ be a nonabelian nested $p$-group with chain of centers $G = X_0 > X_1 > \dots X_n > 1$ and let $d_i^2 = |G:X_i|$ for $i = 0, \dots, n$.  If $G' \le X_n$, then $|[X_{i-1},G]:[X_i:G]| \le d_i$ for all $i = 1,\dots, n$.  In particular, $|G'| \le d_1 \cdots d_n$.
\end{theorem}

\begin{proof}
We have by Lemma \ref{zero} that $X_n = Z(G)$ so $G$ has nilpotence class $2$.  Hence, $G$ is a GVZ group, and applying Theorem \ref{sevena}, we have $|\cd {G/[X_i,G] \mid G/[X_{i-1},G]}| = 1$ for all $i = 1, \dots, n$.  By Corollary \ref{nine4}, we see that $G'$ is an elementary abelian $p$-group for some prime $p$.  Thus, we can find a subgroup $Y$ of $G'$ so that $G' = [X_{i-1},G] Y$ and $[X_{i-1},G] \cap Y = [X_i,G]$.  By Lemma \ref{sixa}, we have $Z(G/Y) = X_i/Y$.  If $\chi \in \irr {G/Y}$, then $[X_i,G] \le Y \le \ker {\chi}$ and if $[X_{i-1},G] \le \ker {\chi}$, then $G' = [X_{i-1},G]Y \le \ker {\chi}$.  Applying Lemma \ref{threeb}, we have that either $\chi$ is linear or $Z (\chi) = X_i$.  Since $G$ is a GVZ group, it follows that every nonlinear irreducible character of $G/Y$ vanishes off of the center of $G/Y$, and thus, $G/Y$ is a VZ-group.  This implies that $|G:X_i| \ge |G':Y|^2$.  Since $|G:X_i| = d_i^2$, we have $d_i \ge |G':Y| = |Y[X_{i-1},G]:Y| = |[X_{i-1},G]:Y \cap [X_{i-1},G]| = |[X_{i-1},G]:[X_i,G]|$.  Notice that $|G'| = |[X_0,G]:[X_1,G]| \cdots |[X_{n-1},G]:[X_n,G]| \le d_1 \cdots d_n$.
\end{proof}

We also have the following amusing result for a nested $p$-group $G$ with nilpotence class $2$: every subset of $\cd G$ containing $1$ occurs as the set of character degrees for some quotient.

\begin{lemma}
Let $G$ be a nonabelian $p$-group and assume $G' \le Z(G)$.  If $S \subseteq \cd G \setminus \{ 1 \}$, then there exists a subgroup $N \le G'$ so that $\cd {G/N} = \{ 1 \} \cup S$.
\end{lemma}

\begin{proof}
We work by induction on $|G|$.  Write $G = X_0 > X_1 > \dots X_n > 1$ for the chain of centers of $G$, and let $d_i^2 = |G:X_i|$ for $i = 0, \dots, n$.  Because $G$ has nilpotence class 2, we know that $G$ is a GVZ group.  By Lemma \ref{strict}, we know that $\cd G = \{ 1, d_1, \dots, d_n \}$. If $S = \{ d_1, \dots, d_n \}$, then take $N = 1$, and the result holds.  Thus, we may assume that $S$ is a proper subset of $\{ d_1, \dots, d_n \}$.  Hence, we can find $d_i \in \{ d_1, \dots, d_n \} \setminus S$.  {}From Corollary \ref{nine4}, we know that $G'$ is elementary abelian.  If $i = n$, then take $Y = [X_{n-1},G]$ and observe that $\cd {G/Y} = \{ 1, d_1, \dots, d_{n-1} \}$.  Thus, we can suppose $i < n$, and we can find a subgroup $Y$ such that $[X_{i-1},G] = [X_i,G]Y$ and $[X_i,G] \cap Y = [X_{i+1},G]$.  Suppose $j$ is an integer so that $1 \le j \le i-1$.  Then $[X_{j-1},G]Y = [X_{j-1},G] > [X_j,G] = [X_j,G]Y$.  On the other hand, if $j$ is an integer so that $i+1 \le j \le n$, then $|[X_{j-1},G] Y:[X_j,G]Y| = |[X_{j-1},G]:[X_j,G] \cap Y| = |[X_{j-1},G]:[X_j,G]| > 1$.  Applying Lemma \ref{ten}, there exists a character $\chi \in \irr {G/Y}$ with $Z(\chi) = X_j$ for all integers $j$ with either $1 \le j \le i-1$ or $i+1 \le j \le n$.  Since $G$ is a GVZ group, this implies that $\cd {G/Y} = \{ 1, d_1, \dots, d_{i-1}, d_{i+1}, \dots, d_n \}$.  In either case, we have $S \subseteq \cd {G/Y} \setminus \{ 1 \}$ and $|G/Y| < |G|$.  By the induction hypothesis, we can find $N$ so that $Y \le N \le G'$ and $\cd {(G/Y)/(N/Y)} = 1 \cup S$.  Since $G/N \cong (G/Y)/(N/Y)$, the result follows.
\end{proof}

We now work to prove conclusion (4) of Theorem \ref{main5} which is included in Theorem \ref{sevenc}.  Recall that a group $G$ is said to be {\it capable} if there exists a group $\Gamma$ so that $G \cong \Gamma/Z(\Gamma)$.  To prove Theorem \ref{sevenc}, we need to know when a VZ-group is capable.  In Theorem 4 of \cite{mann}, Mann proved that if $G$ is a capable semi-extraspecial group, then $|G:Z(G)| = |G:G'| = |G'|^2$.  From Theorem A of \cite{FAMo} and Lemma 2.2 of \cite{GCG}, we know that a group $G$ is a VZ-group if and only if $G$ is isoclinic to a semi-extraspecial group.  We now see that Mann's result can be extended to capable VZ-groups.  In our proof of this lemma we need  Theorem B of \cite{FAMo} which states that if $G$ is a $p$-group where $|G:Z(G)|$ is a square, then every normal subgroup of $G$ either contains $G'$ or is contained in $Z (G)$ if and only if $G$ is a VZ-group.

\begin{lemma} \label{sevenb}
Let $G$ be a $p$-group that is a capable VZ-group.  Then $|G:Z(G)| = |G'|^2$.
\end{lemma}

\begin{proof}
Since $G$ is a VZ-group, we may use Lemma 2.4 of \cite{GCG} to see that $|G'|^2 \le |G:Z(G)|$.  Using the observation that if $G$ is a VZ-group, then $|G:Z (G)|$ is a square and every normal subgroup of $G$ either contains $G'$ or is contained in $Z(G)$, we may use Lemma 6.4 of \cite{FAMo} which states that a group satisfying this condition that is capable must have $|G'|^2 \ge |G:Z(G)|$.  Combining our two inequalities, we obtain the desired equality.
\end{proof}

We now obtain a result regarding certain nested $p$-groups of nilpotence class larger than $2$.  The following theorem includes Theorem \ref{main5} (4).

\begin{theorem} \label{sevenc}
Suppose $G$ is a nested $p$-group with chain of centers $G = X_0 > X_1 > \dots X_n > 1$ and $n \ge 2$.  If $X_2 < X_2G'$, then the following hold:
\begin{enumerate}
\item $|G:X_1| = |X_2G':X_2|^2 = |G':[X_1,G]|^2$ and $X_2 \cap G' = [X_1,G]$.
\item There exists a subgroup $A$ with $X_2 \le A < X_1$ so that $G/A$ is ultraspecial.
\end{enumerate}
\end{theorem}

\begin{proof}
By Lemma \ref{sixd}, we know that $[X_1,G] \le X_2$.  Since $X_2 < X_2G'$, we may apply Theorem \ref{sevena} (3) to see that $G/X_2$ is a VZ-group with center $X_1/X_2$.  Because $X_2/[X_2,G] = Z (G/[X_2,G])$, we conclude that $G/X_2$ is capable.  By Lemma \ref{sevenb}, we have $|G:X_1| = |X_2G':X_2|^2$.  Notice that $[X_1,G] \le X_2 \cap G'$ so $$|G':[X_1,G]| \ge |G':X_2 \cap G'| = |X_2G':X_2|.$$  In light of Theorem \ref{sevena} (4), we have 
$$|G':[X_1,G]| \le |G:X_1|^{1/2} = |X_2G':X_2|.$$  Thus, we must have equality and so, $X_2 \cap G' = [X_1,G]$.

By Lemma \ref{nine}, we know that $X_1/X_2$ is elementary abelian.  We can find a subgroup $A$ with $X_1/X_2 = X_2G'/X_2 \times A/X_2$.  This implies that $X_1 = A(X_2G') = AG'$ and $A \cap X_2G' = X_2$.  Since $X_1/X_2$ is central in $G/X_2$ by Lemma \ref{sixd}, we see that $A$ is normal in $G$.  Observe that $|X_1:A| = |X_2G':X_2| > 1$.  Applying Theorem \ref{sevena} (3), $G/A$ is a semi-extraspecial group.  From conclusion (1), we see that $|X_1:A| = |G:X_1|^{1/2}$, and this implies that $G/A$ is ultra-special.
\end{proof}

\noindent
{\bf Question:}  Do we have any hope of determining when a nested GVZ-group will be capable?

We close this section by considering a nested $p$-group with $p \in \cd G$.

\begin{lemma}
Let $G$ be a nested $p$-group with chain of centers $G = X_0 > X_1 > \dots X_n \ge 1$.  If $p \in \cd G$, then $|G:X_1| = p^2$ and $|G':[X_1,G]| = p$.
\end{lemma}

\begin{proof}
Since $p \in \cd G$, there exists a character $\chi \in \irr G$ with $\chi (1) = p$, and set $N = \ker {\chi}$.   Since $G$ is a $p$-group, we know that $\chi$ is monomial, so there exists a subgroup $T < G$ with a character $\lambda \in \irr T$ so that $\lambda^G = \chi$ and $\lambda (1) = 1$.  It follows that $|G:T| = p$.  Hence, we have that $T$ is normal in $G$, and thus, $T'$ is normal in $G$.  By Lemma 5.11 of \cite{text}, we have that $N = {\rm Core}_G (\ker {\lambda})$.  Since $\ker{\lambda} \le T$, this implies $N \le T$.  As $\lambda$ is linear, we have $T' \le \ker{\lambda}$ and since $T'$ is normal in $G$, we conclude that $T' \le {\rm Core}_G (\ker{\lambda}) =  N$.  Since $T/N$ is a normal abelian subgroup of $G/N$, we may apply It\^o's theorem to see that every degree in $\cd {G/N}$ divides $|G:T| = p$.  Using the fact that $\chi \in \irr {G/N}$, we deduce $\cd {G/N} = \{ 1, p \}$.  Observe that $[G',G]N < G'N$, so $G/[G',G]N$ is a nested group with nilpotence class $2$.  We know by Lemma \ref{sixa} that $Z (G/N) = X_i/N$ for some $i \in \{ 1, \dots, n \}$.  Applying Theorem \ref{sevena}, we have that $|G:X_i| = p^2$.  Applying the fact that $p^2 \le |G:X_1| < |G:X_2|$, we conclude that $i = 1$ and $|G:X_1| = p^2$.  By Theorem \ref{sevena}, we have $|G':[G,X_1]| \le p$, and by Corollary \ref{threec}, $[G,X_1] < [G,X_0] = G'$.  Therefore, $|G':[G,X_1]| = p$.
\end{proof}

\section{Examples}

In this section we present a number of different examples.  We begin by noting that any group with a unique chief series will be nested.  In particular, all simple groups trivially are nested under this definition.  It is not difficult to find nonsolvable examples of nested groups that are not simple.  To see examples of nested groups that are solvable, but not nilpotent, observe that a Frobenius group $G$ where $G'$ is the Frobenius kernel and minimal normal will be a nested group.

Using the small groups database in the computer algebra system \cite{magma}, the group ${\rm SmallGroup} (32,9)$ gives an example of a nested group that is not nested by degrees.  On the other hand, if if $G$ is a simple group, then $G$ is nested by degrees since if $\chi,\nu \in \irr G$ with $\chi (1) \le \nu (1)$, then there are three possibilities: (1) $\nu (1) = 1$ and in this case, $\chi (1) = 1$ so $Z(\chi) = Z (\nu) = G$; (2) $\chi (1) = 1$ and $\nu (1) > 1$ and in this case $G = Z (\chi) \ge Z (\nu) = 1$; or (3) $\chi (1) > 1$, and in this case, $\nu (1) > 1$ and $Z(\chi) = Z (\nu) = 1$.  On the other hand, since a nonabelian simple group has at least three character degrees other than $1$, it will not be strictly nested by degrees.  In addition, if $G$ is a Frobenius group where $G'$ is the Frobenius kernel and $G'$ is minimal normal in $G$ and $G/G'$ is abelian, then $G$ is a solvable group that is strictly nested by degrees.  

At this time, we do not have any examples of solvable or nilpotent groups that are nested by degrees but not strictly nested by degrees.  In fact, Shawn Burkett has written code in Magma to check all of the small groups of order $2^n$ for $n \le 8$ and $3^n$ for $n \le 7$, and no examples have been found.  However, we see nothing preventing the existence of $p$-groups that are nested by degrees that are not strictly nested by degrees.


In \cite{gvz}, Nenciu gives examples of GVZ-groups that are not nested and of GVZ-groups that are nested.  In neither \cite{gvz} nor \cite{nested} does Nenciu present any groups that are nested, but not a GVZ-group.  We rectify that now by noting that ${\rm SmallGroup} (64,258)$ in the small groups library in Magma is an example of such a group.

It is not difficult to see that if $G$ is a $p$-group with $|G| = p^a$, then $G$ has nilpotence class at most $a-1$.  Following standard nomenclature, we say that $G$ is of {\it maximal class} if $G$ has nilpotence class $a-1$.  As we mentioned in the Introduction, Qian and Wang proved in \cite{qiwa} that a $p$-group $G$  has that the kernels of the nonlinear irreducible characters form a chain if and only if $G$ has maximal class or $G'$ has order $p$ and $Z(G)$ is cyclic.  Note that if $G$ has maximal class, then this implies that the centers of the nonlinear irreducible characters form a chain, and by Lemma \ref{knfour}, $G$ is a nested group.  If, in addition, $G$ has order $p^n$ where $n \ge 4$, then $|G:Z_{n-2} (G)| = p^3$ and $|G:Z_{n-3}| = p^4$.  It follows that $\cd {G/Z_{n-2}} = \cd {G/Z_{n-3}} = \{ 1, p\}$, and so we cannot have that $G$ is nested by degrees.  Hence, the maximal class groups of nilpotence class more than $2$ give examples of nested groups that are not nested by degrees. 



In the literature, the idea of maximal class has been generalized to the idea of coclass.  A $p$-group $G$ is said to have {\rm coclass} $n$ if $|G| = p^m$ for some integer $m \ge 2$ and the nilpotence class of $G$ is $m-n$.  In particular, $G$ is of maximal class if and only if $G$ has coclass $1$.  By the result of Qian and Wang, we see that if $G$ has coclass $2$ and either $G'$ does not have order $p$ or the center of $G$ is not cyclic, then the kernels of the nonlinear irreducible characters of $G$ do not form a chain.  On the other hand, we can show that groups of coclass $2$ are nested.

\begin{lemma}\label{last}
If $G$ is a $p$-group of coclass $2$ and nilpotence class at least $3$, then $G$ is a nested group whose chain of centers is the upper central series for $G$.  If in addition $G$ has nilpotence class at least $4$, then $G$ is not nested by degrees.
\end{lemma}

\begin{proof}
We work by induction on the order of $G$.  Let $Z = Z(G)$.  Observe that either $|Z| = p^2$ and $G/Z$ has maximal class or $|Z| = p$ and $G/Z$ has coclass $2$.  In particular, $G/Z$ is nested.  In light of Theorem \ref{main3}, it suffices to show that if $M$ is a normal subgroup of $G$ that does not contain $Z$, then $Z_M$ is a term of the upper central series of $G$ where as before, $Z_M/M = Z(G/M)$.  If $M = 1$, then $Z_M = Z$ and we have what we need.  Thus, we may assume that $M > 1$.  Since $M$ does not contain $Z$, this implies that $|Z| = p^2$.  Notice that this implies that $G/Z$ has maximal class.  Let $L = Z \cap M$, and observe that $1 < L < Z$.  This implies that $|L| = p$.   Observe that $Z \le Z_L$.  Also, $[G,Z_L] \le L < Z$, so $Z_L/Z \le Z(G/Z) = Z_2/Z$.  Hence, we have that $Z \le Z_L \le Z_2$.  Since $G/Z$ has maximal class and is not abelian, we have that $|Z_2:Z| = p$.  We deduce that either $Z_L = Z$ or $Z_L = Z_2$.  Notice that $G/L$ will either have maximal class or will have coclass $2$.  In particular, we can apply the induction hypothesis to $G/L$ to see that $G/L$ will be nested and the chain of centers is the upper central series of $G/L$.  Since $Z_L$ is a member of the upper central series of $G$, the upper central series of $G/L$ is the homomorphic image of the upper central series of $G$.  By Lemma \ref{ten}, we know that $Z_M/L$ will be a term in the upper central series of $G/L$, and thus, $Z_M$ is a term in the central series of $G$, and this proves the lemma.

Suppose that $|G| = p^6$, so $G$ has nilpotence class $4$.  Notice $\cd G \subseteq \{ 1, p, p^2 \}$.  We have seen that if $G$ is nested by degrees, then its nilpotence class is at most the number of character degrees, and so, $G$ cannot be nested by degrees.  If $G$ has nilpotence class at least $4$, then $G$ will have a quotient of order $p^6$.  This quotient will either have maximal class or have coclass $2$.  In either case, the quotient cannot be nested by degrees, and thus, it follow that $G$ cannot be nested by degrees.
\end{proof}

To give an idea of how common or rare nested GVZ-groups are, we present in Table \ref{table 1} the numbers of groups found for orders $2^5$, $2^6$, $2^7$, $2^8$, $3^5$, $3^6$, and $3^7$.  Observe from the table that the hypothesis in the second conclusion of Lemma \ref{last} that $G$ have nilpotence class at least $4$ is necessary.  Note that any groups of the orders given that have class $6$ or class $7$ will be maximal class or coclass $2$, and we just showed that all of these groups will be nested and will not be nested by degrees, so we have not included these columns in the table.  By Lemma \ref{strict}, we see that when $G$ is class $2$ the set of nested groups is the same as the set of groups that are nested by degrees and the same as the set of nested GVZ-groups.  Recall that a group that is nested by degrees has nilpotence class at most the number of distinct character degrees.  If $|G| = p^{2k +1}$ or $|G| = p^{2k+2}$ and $G$ is nested by degrees, then $G$ has nilpotence class at most $k+1$.  This explains why there are no groups that are nested by degrees of class $5$ when $|G|= 2^8$.

\begin{table}[ht]
	\begin{tabular}{|c|c|c|c|c|c|}
		\hline
		Order&Type&Class 2&Class 3&Class 4&Class 5\\
		\hline
		\multirow{5}{*}{$2^5$}&All&26&15&3&0 \\
		\cline{2-6}
		&Nested&17&15&3&0 \\
		\cline{2-6}
		&Nested by&\multirow{2}{*}{17}&\multirow{2}{*}{5}&\multirow{2}{*}{0}&\multirow{2}{*}{0} \\
		&Degrees&&&&\\
		\cline{2-6}
		&Nested GVZ&17&5&0&0 \\
		\hline
		\multirow{5}{*}{$2^6$}&All&117&114&22&3 \\
		\cline{2-6}
		&Nested&57&52&22&3 \\
		\cline{2-6}
		&Nested by&\multirow{2}{*}{57}&\multirow{2}{*}{27}&\multirow{2}{*}{0}&\multirow{2}{*}{0} \\
		&Degrees&&&&\\
		\cline{2-6}
		&Nested GVZ&57&24&0&0 \\
		\hline
		\multirow{5}{*}{$2^7$}&All&947&1137&197&29 \\
		\cline{2-6}
		&Nested&211&172&97&29 \\
		\cline{2-6}
		&Nested by&\multirow{2}{*}{211}&\multirow{2}{*}{95}&\multirow{2}{*}{11}&\multirow{2}{*}{0} \\
		&Degrees&&&&\\
		\cline{2-6}
		&Nested GVZ&211&76&11&0 \\
		\hline
		\multirow{5}{*}{$2^8$}&All&31742&21325&2642&320 \\
		\cline{2-6}
		&Nested&747&698&328&185 \\
		\cline{2-6}
		&Nested by&\multirow{2}{*}{747}&\multirow{2}{*}{371}&\multirow{2}{*}{81}&\multirow{2}{*}{0} \\
		&Degrees&&&&\\
		\cline{2-6}
		&Nested GVZ&747&277&66&0 \\
		\hline
		\multirow{5}{*}{$3^5$}&All&28&26&6&0 \\
		\cline{2-6}
		&Nested&17&26&6&0 \\
		\cline{2-6}
		&Nested by&\multirow{2}{*}{17}&\multirow{2}{*}{6}&\multirow{2}{*}{0}&\multirow{2}{*}{0} \\
		&Degrees&&&&\\
		\cline{2-6}
		&Nested GVZ&17&6&0&0 \\
		\hline
		\multirow{5}{*}{$3^6$}&All&133&282&71&7 \\
		\cline{2-6}
		&Nested&60&142&71&7 \\
		\cline{2-6}
		&Nested by&\multirow{2}{*}{60}&\multirow{2}{*}{58}&\multirow{2}{*}{0}&\multirow{2}{*}{0} \\
		&Degrees&&&&\\
		\cline{2-6}
		&Nested GVZ&60&51&0&0 \\
		\hline
		\multirow{5}{*}{$3^7$}&All&1757&6050&1309&173 \\
		\cline{2-6}
		&Nested&299&990&588&173 \\
		\cline{2-6}
		&Nested by&\multirow{2}{*}{299}&\multirow{2}{*}{222}&\multirow{2}{*}{33}&\multirow{2}{*}{0} \\
		&Degrees&&&&\\
		\cline{2-6}
		&Nested GVZ&299&187&33&0 \\
		\hline
	\end{tabular}
	\caption{Numbers of Groups} \label{table 1}
\end{table}

We now are going to recover a result of Isaacs that is the main result of \cite{isap}.  The main result of that paper is that if $p$ is a prime and ${\mathcal A}$ is a set of powers of $p$ that contain $1$, then there is a $p$-group $P$ of nilpotence class $2$ with $\cd P = {\mathcal A}$.  We now show that $P$ can be chosen to have the additional property that $P$ is nested by degrees.  Note that Theorem \ref{ex1} yields Theorem \ref{main6} of the Introduction.

\medskip
\noindent
{\bf Example 1:}  Suppose ${\mathcal A} = \{ 1, p^{n_1}, \dots, p^{n_m} \}$ so that $0 < n_1 < n_2 < \dots < n_m$.  We let $P$ be the $p$-group generated by $x_1, x_2, \dots, x_{2n_m}$ and $z_1, \dots, z_m$ where each of the $x_i$'s and $z_j$'s have order $p$ and the $z_j$'s are all central.  We assume that we have the following commutators.  For each integer $j$ satisfying $1 \le j \le m$ set the commutators $[x_i,x_{i+n_j}] = z_j$ for each integer $1 \le i \le n_j$.  Assume that the remaining commutators are defined linearly or are trivial.

\begin{theorem}\label{ex1}
Let $G$ be the group defined in Example 1.  Then $G$ is a $p$-group of nilpotence class $2$, $G$ is nested by degrees, and $\cd G = \{ 1, p^{n_1}, \dots, p^{n_m} \}$.
\end{theorem}

\begin{proof}
We see that $G' = \langle z_1, \dots, z_m \rangle \le Z(G)$ and $|G| = p^{2n_m + m}$, so $G$ is a $p$-group of nilpotence class $2$.  For each integer $i$ with $0 \le i \le m$, set $X_i = \langle x_{2n_i + 1}, \dots, x_{2m}, Z(G) \rangle$.  Observe that $X_0 > X_1 > \dots X_m = G' = Z(G)$ is a series of normal subgroups.

Let $N$ be a subgroup of $Z(G)$ and let $i$ be the largest integer in $\{ 1, \dots, m \}$ such that $z_i \not\in N$. Hence, for all $j \ge i+1$, we have $z_j \in N$.  We claim that $X_i/N \le Z (G/N)$.  Consider an integer $j$ with $j \ge 2n_i + 1$.  Since $j > 2n_i$, we see that any commutator of $x_j$ with any $x_k$ will only involve $z_{j+1}, \dots, z_m$ and will lie in $N$.  It follows that $[x_j,G] \subseteq N$.  Since $X_i$ is generated by such $x_j$'s it follows that $[X_i,G] \le N$ and so, $X_i/N \le Z(G/N)$.

Let $N$ be a subgroup of $Z(G)$ having index $p$ in $Z(G)$.  We claim that $Z(G/N) = X_i/N$ where $i$ is the largest integer in $\{ 1, \dots, m \}$ such that $z_i \not\in N$.  Let $Z/N = Z(G/N)$, and by the previous paragraph, we have $X_i \le Z$.  Now, let $Y = \langle  x_{2n_i + 1}, \dots, x_{2m}, N \rangle$.  Observe that $N \le Y \le X_i$.  Since $X_i/N$ is central in $G/N$, this implies that $Y$ is normal in $G$.  Also, we see that $X_i/N = Y/N \times Z(G)/N$.  Let $E_i$ be the extraspecial group generated by $e_1, \dots, e_{2n_i}, z$ where each $e_k$ and $z$ has order $p$, $z$ is central, and $[e_j,e_{i+n_i}] = z$.  Observe that the map defined by $x_k Y \mapsto e_k$ for all $k$ with $1 \le k \le 2n_i$and $z_i Y \mapsto z$ is an isomorphism.  Note that this forces $X_i/Y = Z(G/Y)$.  We see that $[Z,G] \le N \le Y$ and so, $Z/Y \le Z(G/Y) = X_i/Y$.  We deduce that $Z \le X_i$.  Since we already showed that $X_i \le Z$, we conclude that $Z= X_i$.

Let $M$ be any normal subgroup of $G$.  Let $Z/M = Z(G/M)$.  If $Z(G) = G' \le M$, then $Z/M = G/M = X_0/M$ and $Z = X_0$.  Thus, we may assume that $Z(G)$ is not contained in $M$.  Let $i$ be the largest integer in $\{ 1, \dots, m \}$ so that $z_i$ is not contained in $M$.  By the second paragraph, we have $X_i \le Z$.  Let $N$ be a subgroup of index $p$ in $Z(G)$ that contains $M \cap Z(G)$ and does not contain $z_i$.  We see that $[Z,G] \le G' \cap M = Z(G) \cap M \le N$.  Hence, $Z/N \le Z(G/N) = X_i/N$ by the previous paragraph and so, $Z \le X_i$.  Therefore, we have $Z = X_i$.  This shows that $Z(G/[X_i,G]) = X_i/[X_i,G]$.  Applying Lemma \ref{sixca}, we conclude that $G$ is nested with chain of centers $G = X_0 > X_1 > \dots > X_m > 1$.  We have previously observed that any group of nilpotence class $2$ is a GVZ-group.  We obtain the remaining conclusions by appealing to Lemma \ref{strict}.
\end{proof}

\noindent
{\bf Example 2:}  Let $p$ be a prime and let $q$ and $r$ be distinct primes such that $p$ divides $q-1$ and $r-1$.  Let $G$ be the Frobenius group whose Frobenius kernel has order $qr$ and whose Frobenius complement has order $p$. Let $Q$ be the Sylow $q$-subgroup of $G$ and let $R$ be the Sylow $r$-subgroup of $G$.  Observe that $QR = G'$ is the Frobenius kernel of $G$.  Also, $G/Q$ and $G/R$ are Frobenius groups of order $pr$ and $pq$ respectively.  Notice that the nonlinear irreducible characters of $G/Q$ and $G/R$ will be faithful.  Hence, there exist irreducible characters of $G$ whose kernels are $Q$ and $R$, respectively.  Let $1_Q \ne \alpha \in \irr Q$ and $1_R \ne \beta \in \irr R$.  Then $\alpha \times \beta \in \irr {QR}$ and so, $\chi = (\alpha \times \beta)^G \in \irr G$.  Observe that $\ker {\chi} = 1$.  It follows that every proper subgroup of $QR = G'$ lies in ${\rm nlKer} (G)$.  On the other hand, subgroups in ${\rm nlKer} (G)$ do not form a chain.  Hence, the equivalence for $p$-groups found in the Main Theorem of \cite{qiwa} does not hold when the $p$-group hypothesis is dropped.

\medskip
We note that Nenciu includes an example in \cite{nested} of a nested GVZ-group of order $p^{2n+1}$ and nilpotence class $n+1$.  Her example has exponent $p$.  We note that there is a small error in Nenciu's example.  She states in \cite{nested} that $p \ge n+1$; however, for her computations on the exponent to be valid, she actually needs $p > n+1$.  The group constructed when $p = n+1$ has exponent $p^2$ not $p$.  The mistake comes because $\gamma_{n+1} (G) \ne 1$ when $G$ has nilpotence class $n+1$.  

The following example is from the Introduction of \cite{some}.  See also Example 3 in \cite{ChMc}.  This example also has order $p^{2n+1}$ and nilpotence class $n+1$; however, this example has exponent $p^{n+1}$, and the only restriction we need on $p$ is that $p$ is odd.

\medskip
\noindent
{\bf Example 3:}  Let $p$ be an odd prime and let $n$ be a positive integer.  Let $C_n$ be the cyclic group of order $p^{n+1}$. It is well known that the Sylow $p$-subgroup of ${\rm Aut} (C_n)$ is cyclic of order $p^n$.  Write $A_n$ for the Sylow $p$-subgroup of ${\rm Aut} (C_n)$, and let $G_n = C_n \rtimes A_n$.  Observe that $|A_n| = p^n$, and so, $G_n = p^{n+1}p^n = p^{2n+1}$  


Let $C_n = \langle x \rangle$ and $A_n = \langle y \rangle$ so that $x^{p^{n+1}} = 1$ and $y^{p^{n}} = 1$.  We have $x^y = x^{1+p}$ which yields $[x,y] = x^p$.  If $i$ is an integer, then $(x^i)^y = (x^y)^i = (x^{p+1})^i = x^{ip+i}$, and thus, $[x^i,y] = x^{ip}$ for every integer $i$.  Now for a nonnegative integer $j$, we obtain $$[x^iy^j,y] = [x^i,y]^{y^j} [y^j,y] = (x^{ip})^{y^j} = (x^{ip})^{(1+p)^j} = x^{ip(1+p)^j}.$$  If $k$ is a nonnegative integer, then $(x^i)^{y^k} = (x^i)^{(1+p)^k} = x^{i(1+p)^k}$, and we obtain $[x^i,y^k] = x^{i(1+p)^k - i}$.  It follows that $G_n' = \langle x^p \rangle$.  

Observe that $[x^{p^n},y] = x^{p^{n+1}} = 1$, so $x^{p^n} \in Z(G_n)$.    Suppose $x^iy^j \in Z(G_n)$, then by above, we must $p^{n+1}$ divides $ip(1+p)^j$, and this implies that $p^n$ divides $i$.  We have already seen that $x^{p^n}$ is central.  It follows that $y^j \in Z(G)$.  Now, $[x,y^j] = x^{(1+p)^j -1}$.  It follows that $p^{n+1}$ divides $(1+p)^j-1$ and this will occur only if $p^n$ divides $j$.  But $y^{p^n} = 1$, and so, $x^iy^j \in \langle x^{p^n} \rangle$.  We conclude that $Z(G_n) = \langle x^{p^n} \rangle$.  Let $Y = \langle x^{p^n}, y^{p^{n-1}}\rangle$.  It is not difficult to see that $Y$ is normal in $G_n$ and that $G_n/Y \cong G_{n-1}$.  

For all integers $0 \le i \le n+1$, set $X_i = \langle x^{p^i},y^{p^i} \rangle$, $Y_i = \langle x^{p^{i+1}}, y^{p^i} \rangle$, and $Z_i = [X_i,G_n]$.  The above computations show that $Z_i = \langle x^{p^{i+1}} \rangle$.  It follows that $X_i = Z_{i-1} Y_i$ and $Z_i \le X_{i+1} \le Y_i \le X_i$.  Notice that $X_i/Z_i \le Z (G_n/Z_i)$.  The previous paragraph shows that $G_n/Y_i \cong G_i$.  It follows that $Z(G_n/Y_i) = Z_{i-1}Y_i/Y_i = X_i/Y_i$.  Let $Z/Z_i = Z(G_n/Z_i)$.  We see that $[Z,G_n] \le Z_i \le Y_i$, and so, $Z/Z_i \le Z(G_n/Y_i) = X_i/Y_i$.  We obtain $X_i \le Z \le X_i$, and hence, $Z =X_i$.  This implies that $X_i/Z_i = Z(G_n/Z_i)$ for $0 \le i \le n$.  Recall that $G_n' = \langle x^p \rangle$ is cyclic and so $Z_0, Z_1, \dots, Z_n$ is a complete list of the subgroups of $G_n'$.  

Manipulating the commutator equations above, it is not difficult to show that for every $g \in G_n \setminus Z_n$ and for every $z \in Z_n$, there exists $h \in G_n$ so that $[g,h] = z$.  This implies that $(G_n,Z_n)$ is a Camina pair.    It follows that for each $i$ with $1 \le i \le n$ that $(G_n/Y_i,X_i/Y_i)$ is a Camina pair.  

Let $N$ be a normal subgroup of $G_n$ and let $Z_N$ be defined by $Z_N/N = Z(G_n/N)$.  We see that $[Z_N,G_n] \le G_n'$, so $[Z_N,G_n] = Z_i$ for some integer $i$ with $0 \le i \le n$.  Since $X_i/Z_i = Z (G_n/Z_i)$, it follows that $Z_N \le X_i$.  We have $Z_i \le N \le Z_N \le X_i$.   On the other hand, $[X_i,G_n] \le Z_i \le N$, so $X_i/N \le Z(G_n/N) = Z_N/N$.  We deduce that $N < Z_N = X_i$.  Applying Theorem \ref{main3}, we conclude that $G_n$ is a nested group.  

Suppose $\chi \in \irr {G_n}$.  Since $[\ker{\chi},G_n]$ is a subgroup of $G_n'$, we use the observation above to see that there is an integer $i$ so that $[\ker{\chi},G_n] = Z_i$, and so, $Z(\chi) = X_i$ by the previous paragraph.  Also, $Z_i \le \ker {\chi} < X_i$.  We have that $Z(G_n/Z_i) = X_i/Z_i = Z_{i-1}/Z_i \times Y_i/Z_i$.  Let $\mu$ be a constituent of $\chi_{X_i}$.  Then $\mu = \lambda \times \delta$ where $\lambda \in \irr {Z_{i-1}/Z_i}$ and $\delta \in \irr{Y_i/Z_i}$.  If $\lambda = 1$, then $X_i \le \ker {\chi}$; thus, we must have that $\lambda \ne 1$.  We have $\lambda \times 1_{Y_i} \in \irr {X_i/Y_i}$.  Since $(G_n/Y_i,X_i/Y_i)$ is a Camina pair, $(\lambda \times 1_{Y_i})^{G_n}$ has a unique irreducible constituent $\psi \in \irr {G_n/Y_i}$ and $\psi$ is fully ramified with respect to $G/X_i$.

On the other hand, $1_{Z_{i-1}} \times \delta$ extends to $1_{Z_0} \times \delta \in \irr {Z_0Y_i/Z_0}$.  Since $Z_0 = G'$, we see that $1_{Z_0} \times \delta$ will have an extension to $\hat\delta \in \irr {G/G'}$.  It follows that $1_{Z_{i-1}} \times \delta$ extends to $\hat \delta$.  Since $\lambda \times \delta = (\lambda \times 1_{Y_i}) (1_{Z_{i-1}} \times \delta)$, we see that $(\lambda \times \delta)^G = \hat\delta (\lambda \times 1_{Y_i})^G$.  We conclude that $\hat\delta \psi$ is the unique irreducible constituent of $\mu^G$, and so, $\chi = \hat\delta \psi$.  This implies that $\chi$ is fully ramified with respect to $G_n/X_i$.  In particular, we deduce that $G_n$ is a GVZ group. 

Note that the $X_i$'s in reverse order yield the upper central series of $G$, and so, $G$ has nilpotence class $n+1$.  Observe that the exponent of $G$ is $p^{n+1}$.  This shows that there are nested GVZ groups of arbitrarily large nilpotence class and arbitrarily large exponent.  In fact, this shows that the exponent result in Corollary \ref{nine4} cannot be extended any further without additional hypotheses. 

\medskip

We wish to close with an example that shows that under certain circumstances the bound in Theorem \ref{bound} can be met.  To do this, we need to present more information regarding semi-extraspecial groups.  Let $G$ be a semi-extraspecial $p$-group with $|G:G'| = p^{2n}$ and $|G'| = p^m$.  Following either Theorem 3.1 of \cite{some} or Theorem 1.4 of \cite{ver}, there exist $n \times n$ matrices $A_1, \dots, A_m$ where $a^i_{k,l}$ is the $k,l$-entry of $A_i$ and elements $y_1, \dots, y_{2n} \in G$ and  $z_1, \dots, z_m \in Z(G) = G'$ so that $y_s^p \in Z(G)$ and $z_t^p = 1$ for $1 \le s \le 2n$, $1 \le t \le m$, $G = \langle y_1, \dots, y_{2n}\rangle Z(G)$, $Z(G) = \langle z_1, \dots, z_m \rangle$ that satisfy $[y_k,y_l] = \prod_{i=1}^m z_i^{a^i_{k,l}}$ and $0$ is the only singular matrix in ${\rm span} (A_1, \dots, A_m)$.  We note that if $m = n$ and $G$ has an abelian subgroup of order $p^{2n}$, then the $A_i$'s may be chosen so that $a^i_{k,l} = 0$ for all $i$ when $0 \le k,l \le n$.

\medskip
\noindent
{\bf Example 4:}  Let $r$ be a positive integer, and let $m_1, \dots, m_r$ be positive integers so that $m_i \le \frac {m_{i+1}}2$ for all $1 \le i \le r-1$.  For each $i$, let $G_i$ be a ultraspecial group of order $p^{3m_i}$ with an abelian subgroup of order $p^{2m_i}$, and let $A_{i,1},\dots, A_{i,m_i}$ be the $m_i \times m_i$-matrices as above for $G_i$.  Write $a^{i,j}_{k,l}$ for the $k,l$-entry of $A_{i,j}$, and we assume that $a^{i,j}_{k,l} = 0$ when $1 \le k,l \le m_i$.  The group $G$ will be generated by $x_1, \dots, x_{2m_r}$ and $z_{i,j}$ where $1 \le i \le r$ and $1 \le j \le m_i$.  We define the commutators between $x_k$ and $x_l$ as follows.  Of course, $[x_k,x_k] = 1$.  If $k > l$, then $[x_k,x_l] = [x_l,x_k]^{-1}$.  Thus, we may assume $k < l$.  If there exists an integer $i$ so that $m_i < l \le 2m_i$, then $[x_k,x_l] = \prod_{j=1}^{m_i} z_{i,j}^{a^{i,j}_{k,l}}$.  Note from the choice of the $m_i$'s that if such an integer $i$ exists, then it is unique.  If no such integer $i$ exists, then $[x_k,x_l] = 1$.

Please note that we do not know if we can weaken the hypothesis that $m_i \le \frac {m_{i+1}}2$ and still obtain examples.

\begin{theorem}
Let $G$ be the group defined in Example 4.  Then $G$ is a nested group of nilpotence class $2$ so that $\cd G = \{ 1, p^{m_1}, \dots, p^{m_r} \}$ and $|G'| = p^{m_1 + \dots + m_r}$. 
\end{theorem}

\begin{proof}
It is clear from the construction that $G$ has nilpotence class $2$.  Observe that $G' = \langle z_{k,l} \mid 1 \le k \le r, 1 \le l \le m_i \rangle$, so $|G'| = p^{m_1 + \dots + m_r}$.  For each $i$ so that $1 \le i \le r-1$, we set $X_i =  \langle x_{2m_i + 1}, \dots, x_{2m_r} \rangle G'$ with $X_0 = G$ and $X_r = G'$ and $Y_i = \langle x_{2m_i + 1}, \dots, x_{2m_r} \rangle$.  

We obtain $[X_i,G] = [Y_i,G] = \langle z_{k,l} \mid i+1 \le k \le r, 1 \le l \le m_k \rangle$.  Also, set $U_i = \langle z_{i,l} \mid 1 \le k \le m_i \}$.  Let $W$ be a complement for $[X_{i-1},G] = U_i \times [X_i,G]$ in $G'$ and write $V = Y_i W [X_i,G]$.  Observe that $[X_i,G] \le V \le X_i$, and since $X_i/[X_i,G]$ is central in $G/[X_i,G]$, we see that $V$ is normal in $G$.  It is not difficult to see that $G/V$ is isomorphic to $G_i$ and $X_i = V G'$, so $X_i/V = (G/V)' = Z(G/V)$.   Since $G_i$ is an s.e.s. group, it follows that $G/W[X_i,G]$ is a VZ-group with center $X_i/W[X_i,G]$.

In particular, we can take $W_i = \langle z_{k,l} \mid 1 \le k \le i-1, 1 \le l \le m_k \rangle$ and set $V_i =  Y_i W_i [X_i,G]$.   Let $Z/[X_i,G] = Z(G/[X_i,G])$.  It is easy to see that $X_i \le Z$.  On the other hand, $[Z,G] \le [X_i,G] \le V_i$ implies that $ZV_i/V_i \le Z(G/V_i) = X_i/V_i$, and so, $Z \le X_i$.  We conclude that $Z(G/[X_i,G]) = X_i/[X_i,G]$.  Since $[X_r,G] = 1$, we obtain $Z(G) = X_r = G'$. 

Let $N$ be a normal subgroup of $G$ and let $i$ be the largest integer in $\{ 1, \dots, r\}$ so that $[X_{i-1},G] \not\le N$.  It follows $[X_i,G] \le N$, and hence, $X_i/N \le Z(G/N)$.  If $G' \le N$, then $Z(G/N) = G/N$.  Thus, we may assume that $G' \not\le N$.

Suppose that $N < G'$.  Since $[X_{i-1},G] \not\le N$, we have $N \cap [X_{i-1},G] < [X_{i-1},G]$.  Let $A$ be a complement for $N \cap [X_{i-1},G]$ in $N$.  It follows that $A$ is contained in a complement $W$ for $[X_{i-1},G]$ in $G'$.  Observe that $|WN:W| = |N:N \cap W| = |N:A| = |[X_{i-1},G] \cap N| < |[X_{i-1},G]|$.  It follows that $W[X_i,G] \le WN < G'$.  Since $G/W[X_i,G]$ is a VZ-group with center $X_i/W[X_i,G]$, we deduce that $G/N$ is a VZ-group with center $X_i/N$.   

We now remove the supposition that $N < G'$.  Let $M = N \cap G'$, and since $G' \not\le N$, we have that $M < G'$.  By the previous paragraph, we have $Z(G/M) = X_i/M$.  Let $Z/N = Z(G/N)$.  Using the work two paragraphs ago, we have $X_i \le Z$.  On the other hand, $[Z,G] \le N$ and $[Z,G] \le G'$, and so, $[Z,G] \le G' \cap N = M$.  This implies that $Z/M \le Z(G/M) = X_i/M$, and this yields $Z \le X_i$.  We conclude that $Z = X_i$.  It follows from Lemma \ref{sixa} that $G$ is a nested group with chain of centers $G = X_0 > X_1 > \dots > X_r$.  The character degrees follow from Lemma \ref{strict}.
\end{proof}

\end{document}